\documentclass[oneside]{amsart}
\usepackage[utf8]{inputenc}
\usepackage{import}
\input{preamble.sty}

\title{Structure of quasiconvex virtual joins}
\author{Lawk Mineh}

\begin{document}

\begin{abstract}
    Let \(G\) be a relatively hyperbolic group and let \(Q\) and \(R\) be relatively quasiconvex subgroups.
    It is known that there are many pairs of finite index subgroups \(Q' \leqslant_f Q\) and \(R' \leqslant_f R\) such that the subgroup join \(\langle Q', R' \rangle\) is also relatively quasiconvex, given suitable assumptions on the profinite topology of \(G\). 
    We show that the intersections of such joins with maximal parabolic subgroups of \(G\) are themselves joins of intersections of the factor subgroups \(Q'\) and \(R'\) with maximal parabolic subgroups of \(G\).
    As a consequence, we show that quasiconvex subgroups whose parabolic subgroups are almost compatible have finite index subgroups whose parabolic subgroups are compatible, and provide a combination theorem for such subgroups.
\end{abstract}

\maketitle

\section{Introduction}

The notion of a relatively hyperbolic group was proposed by Gromov \cite{Gromov1987} as a generalisation of word hyperbolic groups. 
The concept has been expanded on by various authors \cite{OsinRHG, BowditchRHG, FarbRHG,DS,GrovesManning}.
A group \(G\) is said to be \emph{hyperbolic relative to} a specified collection of \emph{peripheral subgroups} \(\relto\) when \(G\) exhibits hyperbolic-like behaviour away from the subgroups in this collection.
Archetypal examples of relatively hyperbolic groups include fundamental groups of finite volume manifolds of pinched negative curvature and small cancellation quotients of free products, which are hyperbolic relative to their cusp subgroups and the images of the free factors respectively.

In word-hyperbolic groups, finitely generated subgroups may be very ill-behaved in general, so it is often useful to consider the well-behaved class of \emph{quasiconvex subgroups}.
Quasiconvex subgroups play a central role in the theory of hyperbolic groups: they are exactly the finitely generated undistorted subgroups of hyperbolic groups and are themselves hyperbolic.
Analogously, in a relatively hyperbolic group there is the class of \emph{relatively quasiconvex subgroups} which play a similar role to quasiconvex subgroups in hyperbolic groups.
Relatively quasiconvex subgroups are themselves relatively hyperbolic in a way that is compatible with the ambient group.
Finitely generated undistorted subgroups and \emph{parabolic subgroups} (i.e., subgroups conjugate into the peripheral subgroups) form two basic classes of examples of relatively quasiconvex subgroups.

The intersection of two relatively quasiconvex subgroups is always relatively quasiconvex \cite{HruskaRHCG}, though their subgroup join may not be.
In a previous paper the author and Minasyan establish the relative quasiconvexity of joins of finite index subgroups of relatively quasiconvex subgroups, under some hypotheses on the \emph{profinite topology} of the group.
In particular, we often require that our groups are \emph{QCERF}, meaning that all finitely generated relatively quasiconvex subgroups are closed in the profinite topology (see Subsection~\ref{subsec:profinite_topology} for definitions and examples).

\begin{theorem}[{\cite[Theorem 1.2]{MinMin}}]
\label{thm:virtual_join_qc}
    Let \(G\) be a finitely generated group. Suppose that \(G\) is QCERF hyperbolic relative to a collection of double coset separable subgroups and let \(Q, R \leqslant G\) be finitely generated relatively quasiconvex subgroups.
    Then there are finite index subgroups \(Q' \leqslant_f Q\) and \(R' \leqslant_f R\) of \(Q\) and \(R\) respectively such that \(\langle Q', R' \rangle\) is relatively quasiconvex.
\end{theorem}

In fact, the theorem above establishes the existence of many finite index subgroups of \(Q\) and \(R\) whose join is relatively quasiconvex rather than just one pair, though the existential statement is little technical.
In this article we will be interested in the existence of families of finite index subgroups \(Q' \leqslant Q\) and \(R' \leqslant R\) quantified as follows:

\begin{itemize}
    \descitem{E} there exists \(L \leqslant_f G\) with \(Q \cap R \subseteq L\) such that for any \(L' \leqslant_f L\) satisfying \(Q \cap R \subseteq L'\), there exists \(M \leqslant_f L'\) with \(Q \cap L' \subseteq M\) such that for any \(M' \leqslant_f M\) satisfying \(Q \cap L' \subseteq M'\), we can choose \(Q' = Q \cap M'\) and  \(R'=R \cap M' \leqslant_f R\).
\end{itemize}

The reader less interested in technicalities may roughly interpret \descref{E} as expressing that there exist sufficiently many finite index subgroups \(Q' \leqslant_f Q\) and \(R' \leqslant_f R\) for most purposes.

A relatively quasiconvex subgroup of a relatively hyperbolic group \(G\) is itself relatively hyperbolic with respect to its infinite intersections with maximal parabolic subgroups of \(G\) \cite{HruskaRHCG}.
It is natural to ask, then, about the structure of maximal parabolic subgroups of the relatively quasiconvex joins \(\langle Q', R' \rangle\) obtained from Theorem~\ref{thm:virtual_join_qc}.
The main goal of this paper is to establish that \(Q'\) and \(R'\) can be chosen such that the intersection of \(\langle Q', R' \rangle\) with a maximal parabolic subgroup of \(G\) is, up to conjugacy, itself a join of maximal parabolic subgroups of \(Q'\) and \(R'\).

\begin{theorem}
\label{thm:parab_join}
    Under the conditions of Theorem~\ref{thm:virtual_join_qc}, there is a family of pairs of finite index subgroups \(Q' \leqslant_f Q\) and \(R' \leqslant_f R\) as in \descref{E} such that the following is true. 
    
    Suppose that \(P \leqslant G\) is a maximal parabolic subgroup with \(\langle Q', R' \rangle \cap P\) infinite. 
    Then there is \(u \in \langle Q', R' \rangle\) such that
    \[
        \langle Q', R' \rangle \cap P = u \langle Q' \cap K, R' \cap K \rangle u^{-1},
    \]
    where \(K = u^{-1} P u\).
\end{theorem}

In fact, we obtain a stronger -- though more technical -- characterisation of \(\langle Q', R' \rangle \cap P\) below in Theorem~\ref{thm:parab_join_better}.
Note that the conjugator \(u\) in the above statement is strictly necessary: suppose \(K \leqslant G\) is a maximal parabolic subgroup of \(G\) such that either \(Q' \cap K\) or \(R' \cap K\) is infinite.
Then for any \(v \in \langle Q', R' \rangle\), the intersection \(\langle Q', R' \rangle \cap vKv^{-1}\) contains \(v(Q' \cap K)v^{-1}\) and \(v(R' \cap K)v^{-1}\), and is therefore infinite.
However, it may be that \(u \in \langle Q', R' \rangle\) is such that the subgroups \(Q' \cap P\) and \(R' \cap P\) are both trivial, where \(P = u^{-1} K u\).
This precludes the possibility that they generate \(\langle Q', R' \rangle \cap P\).

Theorem~\ref{thm:parab_join} is a natural extension of a result of Mart\'{i}nez-Pedroza, which states that the intersections \(\langle Q', R' \rangle \cap P\) are conjugate into either \(Q'\) or \(R'\) in the special case that \(R\) is a parabolic subgroup of \(G\) \cite{MPComb}.

Given a maximal parabolic subgroup \(P \leqslant G\), we say that \(Q\) and \(R\) are \emph{compatible at \(P\)} if \(Q \cap P \leqslant R \cap P\) or \(R \cap P \leqslant Q \cap P\), and are \emph{almost compatible at \(P\)} if \(Q \cap R \cap P\) has finite index in either \(Q \cap P\) or \(R \cap P\).
If \(Q\) and \(R\) are (almost) compatible at every maximal parabolic subgroup \(P \leqslant G\), then we say that \(Q\) and \(R\) have \emph{(almost) compatible parabolics}.
The notion of almost compatible parabolics was introduced by Baker and Cooper in the setting of discrete subgroups of \(\operatorname{Isom}(\mathbb{H}^n)\) \cite{Baker_Cooper}. 
We may now state the stronger version of Theorem~\ref{thm:parab_join}.

\begin{theorem}
\label{thm:parab_join_better}
    Let \(G\) be a finitely generated QCERF relatively hyperbolic group, and let \(Q, R \leqslant G\) be finitely generated relatively quasiconvex subgroups.
    Suppose that either \(Q\) and \(R\) have almost compatible parabolics or that each peripheral subgroup of \(G\) is double coset separable.    
    Then there is a finite set \(\mathcal{K}\) of maximal parabolic subgroups of \(G\) and a family of pairs of finite index subgroups \(Q' \leqslant_f Q\) and \(R' \leqslant_f R\) as in \descref{E} such that the following is true.
    
    Suppose that \(P \leqslant G\) is a maximal parabolic subgroup with \(\langle Q', R' \rangle \cap P\) infinite.
    Then there is an element \(u \in \langle Q', R' \rangle\) such that either
    \begin{enumerate}[label=(\roman*)]
        \item \(\langle Q', R' \rangle \cap P = u Q' u^{-1} \cap P\) or,
        \item \(\langle Q', R' \rangle \cap P = u R' u^{-1} \cap P\) or,
        \item \(\langle Q', R' \rangle \cap P =  u\langle Q' \cap K, R' \cap K \rangle u^{-1}\) where \(K = u^{-1} P u\) is an element of \(\mathcal{K}\), and \(Q'\) and \(R'\) are not almost compatible at \(K\).
    \end{enumerate}
    
    Moreover, if either \(Q' \cap P\) or \(R' \cap P\) is infinite, then we may take \(u = 1\) in cases (i) and (ii), and \(u \in Q' \cup R'\) in case (iii).
\end{theorem}

We note that the set \(\mathcal{K}\) is independent of the particular finite index subgroups \(Q'\) and \(R'\).
As an application of Theorem~\ref{thm:parab_join_better}, we show that the condition of having almost compatible parabolics can be virtually promoted to that of having compatible parabolics.

\begin{corollary}
\label{cor:almost_compat->virtual_compat}
    Let \(G\) be a finitely generated QCERF relatively hyperbolic group.
    Suppose that \(Q, R \leqslant G\) are finitely generated relatively quasiconvex subgroups with almost compatible parabolics.
    There are finite index subgroups \(Q' \leqslant_f Q\) and \(R' \leqslant_f R\) such that \(Q'\) and \(R'\) have compatible parabolics.
\end{corollary}

In some special cases, Theorem~\ref{thm:virtual_join_qc} was known before \cite{MinMin}: see \cite{MPComb,Yang,MPS,McCl}.
The extra assumptions appearing in each of these cases imply the condition that \(Q\) and \(R\) have almost compatible parabolics.
Moreover, in these cases it was determined that the joins \(\langle Q', R' \rangle\) decompose as an amalgamated free product.
Using Corollary~\ref{cor:almost_compat->virtual_compat}, we unify and generalise these results as follows.

\begin{corollary}
\label{thm:almost_compat_amalgamation}
    Let \(G\) be a finitely generated QCERF relatively hyperbolic group. 
    Suppose that \(Q, R \leqslant G\) are finitely generated relatively quasiconvex subgroups with almost compatible parabolics.
    Then there are finite index subgroups \(Q' \leqslant_f Q\) and \(R' \leqslant_f R\) such that \(\langle Q', R' \rangle\) is relatively quasiconvex and
    \[
        \langle Q', R' \rangle \cong Q' \ast_{Q' \cap R'} R'.
    \]
\end{corollary}

In general, almost compatibility is a necessary condition for Corollary~\ref{thm:almost_compat_amalgamation} to hold.
Indeed, if \(Q\) and \(R\) are subgroups of the same abelian peripheral subgroup that do not have almost compatible parabolics, then no pair of finite index subgroups \(Q' \leqslant_f Q\) and \(R' \leqslant_f R\) will generate an amalgamated free product over their intersection. 
Corollary~\ref{thm:almost_compat_amalgamation} was known in the case when \(G\) is a discrete subgroup of \(\operatorname{Isom}(\mathbb{H}^n)\) and \(Q\) and \(R\) are geometrically finite subgroups of \(G\) \cite{Baker_Cooper}.

A relatively quasiconvex subgroup of \(G\) is said to be \emph{strongly relatively quasiconvex} if its intersection with each maximal parabolic subgroup of \(G\) is finite, and \emph{full} if its intersection with each maximal parabolic subgroup of \(G\) is either finite or has finite index in that parabolic.
Strongly relatively quasiconvex subgroups are necessarily hyperbolic \cite[Theorem 4.16]{OsinRHG}.
Note that if either of \(Q\) and \(R\) are strongly quasiconvex or full, then they have almost compatible parabolics.
As a consequence of Theorem~\ref{thm:parab_join_better} one obtains the analogue of Theorem~\ref{thm:virtual_join_qc} for each of these types of subgroups.

\begin{corollary}
\label{cor:full_strongly_qc_join}
    Let \(G\) be a finitely generated QCERF relatively hyperbolic group, and let \(Q\) and \(R\) be finitely generated relatively quasiconvex subgroups.
    
    If \(Q\) and \(R\) are strongly (respectively, full) relatively quasiconvex subgroups, then there is a family of pairs of finite index subgroups \(Q' \leqslant_f Q\) and \(R' \leqslant_f R\) as in \descref{E} such that \(\langle Q', R' \rangle\) is also strongly (respectively, full) relatively quasiconvex.     
\end{corollary}

\begin{remark}
    In the special case when either \(Q\) or \(R\) are full quasiconvex subgroups, a version of Theorem~\ref{thm:parab_join} appears in unpublished preprint of Yang \cite{Yang}, claiming that every parabolic subgroup of \(\langle Q', R' \rangle\) is conjugate into either \(Q'\) or \(R'\) in this case. The statement of Corollary~\ref{cor:full_strongly_qc_join} for full quasiconvex subgroups also follows from this.
\end{remark}

This paper is organised as follows.
Section~\ref{sec:prelim} contains the notation and terminology used in this paper and collects preliminary results. 
In Section~\ref{sec:parab} we prove a general result on controlling the parabolic subgroups of relatively quasiconvex subgroups.
In Sections~\ref{sec:shortcutting} and \ref{sec:parab_paths} we recall the construction of a \emph{shortcutting} of a broken line and study properties of paths whose labels represent parabolic elements.
Section~\ref{sec:path_reps} recalls and generalises the terminology relating to \emph{path representatives} as developed in \cite{MinMin}.
Sections~\ref{sec:short_conj} and \ref{sec:end} comprise the proofs of the main results.

\subsection{Acknowledgements}

The author would like to thank Ashot Minasyan for many helpful discussions during the writing of the paper, Eduardo Mart\'{i}nez-Pedroza for comments that helped improve the clarity of exposition, and the anonymous referee for their careful reading and corrections.


\section{Preliminaries}
\label{sec:prelim}

In this section we will establish our use of notation, define relative hyperbolicity, and introduce the basic terminology required in this paper.
Along the way, we collect some auxiliary results.

\subsection{Notation and terminology}

We write \(\NN\) for the set of natural numbers \(\{1, 2, 3 \dots\}\), and \(\NN_0\) for \(\NN \cup \{0\}\).

Let \(G\) be a group.
If \(H\) is a finite index (respectively, finite index normal) subgroup of \(G\), then we write \(H \leqslant_f G\) (respectively, \(H \lhd_f G\)). 

Let \(\mathcal{A}\) be a set with a map \(\mathcal{A} \to G\).
We say that \(\mathcal{A}\) is a \emph{generating set} for \(G\) if its image under this map generates \(G\).
We denote by \(\abs{g}_{\mathcal{A}}\) the length of the shortest word in \(\mathcal{A}^{\pm 1}\) representing \(g\) in $G$, letting \(\abs{g}_{\mathcal{A}} = \infty\) when there is no such word.
If \(\mathcal{A}\) is a generating set for \(G\), then we denote by \(\Gamma(G,\mathcal{A})\) the (left) Cayley graph of \(G\) with respect to \(\mathcal{A}\). 
The standard edge path length metric on \(\Gamma(G,\mathcal{A})\) will be denoted $d_{\mathcal{A}}(\cdot,\cdot)$. 
After identifying \(G\) with the vertex set of \(\Gamma(G,\mathcal{A})\), this metric induces the \emph{word metric} associated to $\mathcal{A}$: \(d_{\mathcal{A}}(g,h) = \abs{g^{-1}h}_{\mathcal{A}}\) for all $g,h \in G$.

Abusing the notation, we will identify the combinatorial Cayley graph $\Gamma(G,\mathcal{A})$ with its geometric realisation. 
The latter is a geodesic metric space and, given two points $x,y$ in this space, we will use $[x,y]$ to denote a geodesic path from $x$ to $y$ in $\Gamma(G,\mathcal{A})$. 
In general $\Gamma(G,\mathcal{A})$ need not be uniquely geodesic, so there will usually be a choice for $[x,y]$, which will either be specified or will be clear from the context (e.g., if $x$ and $y$ already belong to some geodesic path under discussion, then $[x,y]$ will be chosen as the subpath of that path).

Suppose that $p$ is a combinatorial path (edge path) in $\Gamma(G,\mathcal{A})$.
We will denote the initial and terminal endpoints of \(p\) by \(p_-\) and \(p_+\) respectively.
We will write \(\ell(p)\) for the length (i.e., the number of edges) of \(p\). 
We will also use $p^{-1}$ to denote the inverse of $p$, which is the path starting at $p_+$, ending at $p_-$ and traversing $p$ in the reverse direction.
If \(p_1, \dots, p_n\) are combinatorial paths with \((p_i)_+ = (p_{i+1})_-\), for each \(i \in \{1, \dots, n-1\}\), we will denote their concatenation by \(p_1 \dots p_n\).

Since $\Gamma(G,\mathcal{A})$ is a labelled graph, every combinatorial path $p$ comes with a label, which is a word over the alphabet $\mathcal{A}^{\pm 1}$.
We denote by \(\elem{p} \in G\) the element of \(G\) represented by the label of \(p\).
Finally, we write \(\abs{p}_{\mathcal{A}} = |\elem{p}|_{\mathcal{A}}=d_{\mathcal{A}}(p_-,p_+)\). 
Note that the label of $p^{-1}$ is the formal inverse of the label of $p$, so that $|p^{-1}|_{\mathcal{A}}=|p|_{\mathcal{A}}$ and $\widetilde{p^{-1}}={\elem{p}}^{-1}$.

\begin{definition}[Broken line] 
\label{def:broken_line}
    A \emph{broken line} in $\Gamma$ is a path $p$ which comes with a fixed decomposition as a concatenation of combinatorial geodesic paths \(p_1,\dots,p_n\) in \(\Gamma\), so that \(p=p_1p_2 \dots p_n\). 
    The paths \(p_1, \dots, p_n\) will be called the \emph{segments} of the broken line \(p\), and the vertices \(p_-=(p_1)_-, (p_1)_+=(p_2)_-, \dots, (p_{n-1})_+=(p_n)_-\) and \((p_{n+1})_+=p_+\) will be called the \emph{nodes} of \(p\).
\end{definition}

\begin{remark}
    Any combinatorial subpath of a broken line \(p\) is again a broken line, with the decomposition inherited from \(p\).
    Moreover, the concatenation of broken lines is also a broken line in the obvious way.
    We will freely use these facts without reference throughout the paper.
\end{remark}

\begin{definition}
    Let \(\lambda \geq 1\) and \(c \geq 0\). A combinatorial path \(p\) in \(\Gamma(G,\mathcal{A})\) is called (\(\lambda, c\))\emph{-quasigeodesic} if for every combinatorial subpath \(q\) of \(p\), we have
    \[
        \ell(q) \leq \lambda d_{\mathcal{A}}(q_-, q_+) + c.
    \]
\end{definition}

We will make use of the following elementary fact.

\begin{lemma}
\label{lem:almost_containment_inf_subgps}
    Let \(G\) be an infinite group and let \(H, K \leqslant G\) be infinite subgroups.
    If all but finitely many elements of \(H\) are contained in \(K\), then \(H \subseteq K\).
\end{lemma}

\begin{proof}
    Suppose that \(H \setminus K\) is finite, so that its complement (in \(H\)) \(H \cap K\) is infinite.
    Let \(g \in H \setminus K\).
    As \(H \setminus K\) is finite and \(H \cap K\) is infinite, there is some \(h \in H \cap K\) such that \(hg \notin H \setminus K\).
    That is to say, \(hg \in H \cap K\).
    It follows that \(g = (h^{-1})(hg) \in H \cap K\), a contradiction.
    Thus \(H \setminus K\) must be empty and \(H \subseteq K\) as required.
\end{proof}

\subsection{Relatively hyperbolic groups}

\begin{definition}[Relative generating set, relative presentation] 
\label{def:rel_gen_set}
    Let \(G\) be a group, \(X \subseteq G\) a subset and \(\lbrace H_\nu \, | \, \nu \in \Nu \rbrace\) a collection of subgroups of \(G\). 
    The group \(G\) is said to be \emph{generated by \(X\) relative to \(\lbrace H_\nu \, | \, \nu \in \Nu \rbrace\)} if it is generated by \(X \sqcup \mathcal{H}\), where \(\mathcal{H}= \bigsqcup_{\nu \in \Nu} (H_\nu \setminus\{1\})\) (with the obvious map $X \sqcup \mathcal{H} \to G$). 
        If this is the case, then there is a surjection
    \[
        F = F(X) \ast (\ast_{\nu \in \Nu} H_\nu) \to G,
    \]
    where \(F(X)\) denotes the free group on \(X\).
    Suppose that the kernel of this map is the normal closure of a subset \(\mathcal{R} \subseteq F\). Then \(G\) can equipped with the \emph{relative presentation}
\begin{equation} \label{eq:rel_pres}
\langle X, H_\nu, \nu \in \mathcal{N} \mid \mathcal{R} \rangle.    
\end{equation} 

If \(X\) is a finite set, then \(G\) is said to be \emph{finitely generated relative to \(\lbrace H_\nu \, | \, \nu \in \Nu \rbrace\)}. If \(\mathcal{R}\) is also finite, \(G\) is said to be \emph{finitely presented relative to \(\lbrace H_\nu \, | \, \nu \in \Nu \rbrace\)} and the presentation above is a \emph{finite relative presentation}.
\end{definition}

With the above notation, we call the Cayley graph \(\relcay\) the \emph{relative Cayley graph} of \(G\) with respect to  \(X\) and \(\lbrace H_\nu \, | \, \nu \in \Nu \rbrace\).
Note that when \(X\) is itself a generating set of \(G\), \(d_{X\cup\mathcal{H}}(g,h) \leq d_X(g,h)\), for all \(g,h \in G\).

\begin{definition}[Relative Dehn function]
    Suppose that \(G\) has a finite relative presentation \eqref{eq:rel_pres} with respect to a collection of subgroups \(\lbrace H_\nu \, | \, \nu \in \Nu \rbrace\).
    If \(w\) is a word in the free group \(F(X\sqcup\mathcal{H})\), representing the identity in \(G\), then it is equal in \(F\) to a product of conjugates
    \[
        w \stackrel{F}{=} \prod_{i=1}^n a_i r_i a_i^{-1},
    \]
    where \(a_i \in F\) and \(r_i \in \mathcal{R}\), for each \(i\).
    The \emph{relative area} of the word \(w\) with respect to the relative presentation, \(Area^{rel}(w)\), is the least number \(n\) among products of conjugates as above that are equal to \(w\) in \(F\).

    A \emph{relative isoperimetric function} of the above presentation is a function \(f \colon \NN \to \NN\) such that \(Area^{rel}(w)\) is at most \( f(\abs{w})\), for every freely reduced word \(w\) in \(F(X\sqcup\mathcal{H})\) representing the identity in \(G\).
    If an isoperimetric function exists for the presentation, the smallest such function is called the \emph{relative Dehn function} of the presentation.
\end{definition}

\begin{definition} [Relatively hyperbolic group]
\label{def:rh_gp}
    Let \(G\) be a group and let \(\lbrace H_\nu \, | \, \nu \in \Nu \rbrace\) be a collection of subgroups of \(G\). 
    If \(G\) admits a finite relative presentation with respect to this collection of subgroups which has a well-defined linear relative Dehn function, it is called \emph{hyperbolic relative to} \(\lbrace H_\nu \, | \, \nu \in \Nu \rbrace\).
    When it is clear what the relevant collection of subgroups is, we refer to \(G\) simply as a \emph{relatively hyperbolic group}.
    The groups \(\lbrace H_\nu \, | \, \nu \in \Nu \rbrace\) are called the \emph{peripheral subgroups} of the relatively hyperbolic group \(G\), and their conjugates in \(G\) are called \emph{maximal parabolic subgroups}. 
    Any subgroup of a maximal parabolic subgroup is said to be \emph{parabolic}.
\end{definition}

\begin{lemma}[{\cite[Corollary 2.54]{OsinRHG}}]
\label{lem:Cayley_graph-hyperbolic}
    Suppose that \(G\) is a group generated by a finite set \(X\) and hyperbolic relative to a collection of subgroups \(\lbrace H_\nu \mid \nu \in \Nu \rbrace\), and let \(\mathcal{H} = \bigsqcup_{\nu \in \Nu} (H_\nu \setminus \{1\})\). 
    Then the Cayley graph $\relcay$ is $\delta$-hyperbolic, for some $\delta \ge 0$.
\end{lemma}

In order to understand the structure of paths in \(\relcay\) it will be important to examine the behaviour of subpaths labelled by elements of \(\mathcal{H}\).
We collect the necessary definitions and facts for our analysis below.

\begin{definition}[Path components]
    Let \(p\) be a combinatorial path in \(\Gamma(G,X\cup\mathcal{H})\).
    A non-trivial combinatorial subpath of \(p\) whose label consists entirely of elements of \(H_\nu \setminus \{1\}\), for some \(\nu \in \Nu\), is called an \emph{\(H_\nu\)-subpath} of \(p\).

    An \(H_\nu\)-subpath is called an \emph{\(H_\nu\)-component} if it is not contained in any strictly longer \(H_\nu\)-subpath.
    We will call a subpath of \(p\) an \(\mathcal{H}\)-subpath (respectively, an \emph{\(\mathcal{H}\)-component}) if it is an \(H_\nu\)-subpath (respectively, an \(H_\nu\)-component), for some \(\nu \in \Nu\).
\end{definition}

\begin{lemma}[{\cite[Lemma 5.10]{MinMin}}]
\label{lem:rel_paths_with_short_comps}
    Let \(p\) be a path in \(\Gamma(G,X\cup\mathcal{H})\) and suppose there is a constant \(\Theta \geq 1\) such that for any \(\mathcal{H}\)-component \(h\) of \(p\), we have \(|h|_X \leq \Theta\). Then \(|p|_X \leq \Theta \ell(p)\).
\end{lemma}

\begin{definition}[Connected and isolated components]
    Let \(p\) and \(q\) be edge paths in \(\relcay\) and suppose that \(s\) and \(t\) are \(H_\nu\)-subpaths of \(p\) and \(q\) respectively, for some \(\nu \in \Nu\).
    We say that \(s\) and \(t\) are \emph{connected} if \(s_-\) and \(t_-\) belong to the same left coset of \(H_\nu\) in \(G\). 
    This means that for all vertices \(u\) of \(s\) and \(v\) of \(t\) either \(u=v\) or there is an edge \(e\) in \(\relcay\) labelled by an element of \(H_\nu \setminus\{1\}\) with \(e_- = u\) and \(e_+ = v\).

    If \(s\) is an \(H_\nu\)-component of a path \(p\) and \(s\) is not connected to any other \(H_\nu\)-component of \(p\) then we say that \(s\) is \emph{isolated} in \(p\).
\end{definition}

\begin{definition}[Phase vertex]
    A vertex \(v\) of a combinatorial path \(p\) in \(\Gamma(G,X\cup\mathcal{H})\) is called \emph{non-phase} if it is an interior vertex of an \(\mathcal{H}\)-component of \(p\) (that is, if it lies in an \(\mathcal{H}\)-component which it is not an endpoint of).
    Otherwise \(v\) is called \emph{phase}.
\end{definition}

\begin{definition}[Backtracking]
    If all \(\mathcal{H}\)-components of a combinatorial path \(p\) are isolated, then \(p\) is said to be \emph{without backtracking}.
    Otherwise we say that \(p\) \emph{has backtracking}.
\end{definition}

\begin{remark}
\label{rem:comp_of_geod_is_an_edge}
    If \(p\) is a geodesic edge path in \(\Gamma(G,X\cup\mathcal{H})\) then every $\mathcal{H}$-component of $p$ will consist of a single edge, labelled by an element from $\mathcal{H}$. Therefore every vertex of $p$ will be phase.
    Moreover, it is easy to see that \(p\) will be without backtracking.
\end{remark}

\begin{lemma}[{\cite[Lemma 5.12]{MinMin}}]
\label{lem:qgds_with_long_comps}
    For any \(\lambda \geq 1\), \(c \geq 0\) and \(A \geq 0\) there is a constant \(\xi = \xi(\lambda,c,A) \geq 0\) such that the following is true.

    Suppose that \(p\) is a \((\lambda,c)\)-quasigeodesic path in $\relcay$ with an isolated \(\mathcal{H}\)-component \(h\) such that \(\abs{h}_X \geq \xi\). Then \(\abs{p}_X \geq A\).
\end{lemma}

\begin{lemma}
\label{lem:len_of_subgeodesic}
    There is a constant \(\xi_0 \geq 1\) such that if \(v\) is vertex of a geodesic \(p\) in \(\relcay\), then \(d_X(p_-,v) \leq \xi_0 \abs{p}^2_X\).
\end{lemma}

\begin{proof}
    For \(\lambda \geq 1\) and \(A \geq 0\), the constant \(\xi(\lambda,0,A)\) of Lemma~\ref{lem:qgds_with_long_comps} may be taken to be a multiple of \(A\) that depends only on \(\lambda\) (see equation (5.4) in the proof of \cite[Lemma 5.12]{MinMin}).
    Thus there is \(\xi_0 \geq 1\) such that \(\xi(1,0,\abs{p}_X) = \xi_0 \abs{p}_X\).
    Now an application of Lemma~\ref{lem:qgds_with_long_comps} tells us that if \(h\) is an \(\mathcal{H}\)-component of \(p\), then \(\abs{h}_X \leq \xi(1,0,\abs{p}_X) = \xi_0 \abs{p}_X\).
    
    Finally, noting that there are at most \(\abs{p}_{X\cup\mathcal{H}} \leq \abs{p}_X\) edges of \(p\) between \(p_-\) and \(v\) gives that \(d_X(p_-, v) \leq \xi_0 \abs{p}_X^2\) as required.
\end{proof}

In dealing with backtracking in broken lines, we make use of some more specialised terminology.

\begin{definition}[Consecutive backtracking]
    Let \(p=p_1 \dots p_n\) be a broken line in \(\Gamma(G,X\cup\mathcal{H})\).
    Suppose that for some \(i,j\), with \(1 \le i <j \le n\), and \(\nu \in \Nu\) there exist pairwise connected \(H_\nu\)-components \(h_i,h_{i+1},\dots, h_j\) of the paths \(p_i,p_{i+1}, \dots, p_j\), respectively.
    Then we will say that \(p\) has \emph{consecutive backtracking} along the components \(h_i,\dots,h_j\) of \(p_i, \dots, p_j\).
\end{definition}

A key property of relatively hyperbolic groups is that pairs of quasigeodesics in \(\relcay\) whose initial and terminal vertices are close fellow travel (with respect to a proper metric) even when passing through cosets of the peripheral subgroups.
In this paper, we use the below formulation of this property.

\begin{definition}[$k$-similar paths]
    Let \(p\) and \(q\) be paths in \(\Gamma(G,X\cup\mathcal{H})\), and let \(k \geq 0\).
    The paths \(p\) and \(q\) are said to be \emph{\(k\)-similar} if \(d_X(p_-,q_-) \leq k\) and \(d_X(p_+,q_+) \leq k\).
\end{definition}

\begin{proposition}[{\cite[Proposition 3.15, Lemma 3.21 and Theorem~3.23]{OsinRHG}}]
\label{prop:bcp}
    For any \(\lambda \geq 1\), \(c, k \geq 0\) there is a constant \(\kappa = \kappa(\lambda,c,k) \geq 0\) such that if \(p\) and \(q\) are $k$-similar \((\lambda,c)\)-quasigeodesics in \(\Gamma(G,X\cup\mathcal{H})\) and \(p\) is without backtracking, then
    \begin{enumerate}
        \item for every phase vertex \(u\) of \(p\), there is a phase vertex \(v\) of \(q\) with \(d_X(u,v) \leq \kappa\);
        \item every \(\mathcal{H}\)-component \(s\) of \(p\), with \(\abs{s}_X \geq \kappa\), is connected to an \(\mathcal{H}\)-component of \(q\).
    \end{enumerate}
    Moreover, if \(q\) is also without backtracking then
    \begin{enumerate}[resume]
        \item if \(s\) and \(t\) are connected \(\mathcal{H}\)-components of \(p\) and \(q\) respectively, then
        \[
            \max \{ d_X(s_-,t_-), d_X(s_+,t_+) \} \leq \kappa .
        \]
    \end{enumerate}
\end{proposition}

\begin{proposition}[{\cite[Proposition 3.2]{OsinFilling}}]
\label{prop:osin_polygon}
    There is a finite set \(\Omega \subseteq G\) and a constant \(L \ge 0\) such that if \(P\) is a geodesic \(n\)-gon in \(\relcay\) and some side \(p\) is an isolated \(\mathcal{H}\)-component of \(P\) then \(\abs{p}_\Omega \leq Ln\).
\end{proposition}

\begin{remark}
    The previous result does not require that \(G\) is finitely generated.
    When \(G\) is finitely generated we can always choose the generating set \(X\) such that \(\Omega \subseteq X\).
    In this setting \(\abs{p}_X \leq \abs{p}_\Omega\), so we will replace the conclusion of the above with \(\abs{p}_X \leq Ln\).
\end{remark}

\begin{definition}(Relative quasiconvexity)
    A subgroup \(Q \leqslant G\) is said to be \emph{relatively quasiconvex} with respect to \(\relto\) if there is \(\varepsilon \geq 0\) such that for any vertex \(v\) of a geodesic in \(\relcay\) with endpoints in \(Q\), we have \(d_X(v,Q) \leq \varepsilon\).
    Moreover, we will call any such \(\varepsilon \geq 0\) a \emph{quasiconvexity constant} of \(Q\).
\end{definition}

The following is a well-known elementary fact; a proof may be found in \cite[Lemma 5.22]{MinMin}.

\begin{lemma}
\label{lem:props_of_qc_subgroups}
    If \(Q \leqslant G\) is relatively quasiconvex, so is any conjugate or finite index subgroup of \(Q\).
\end{lemma}

\begin{lemma}
\label{lem:hyperbolic_parabolic_quasiconvex}
    Let \(Q \leqslant G\) be a hyperbolic relatively quasiconvex subgroup of \(G\).
    Then for any maximal parabolic subgroup \(P \leqslant G\), the intersection \(Q \cap P\) is quasiconvex in \(Q\).
    In particular, \(Q \cap P\) is hyperbolic.
\end{lemma}

\begin{proof}
    Recall that \(Q\) is hyperbolic relative to a collection of \(Q\)-conjugacy class representatives of infinite subgroups of the form \(Q \cap H\), where \(H \leqslant G\) is a maximal parabolic subgroup of \(G\) \cite[Theorem 9.1]{HruskaRHCG}.
    Thus if \(Q \cap P\) is infinite, it is a maximal parabolic subgroup of \(Q\) and is undistorted in \(Q\) by \cite[Lemma 5.4]{OsinRHG}.
    It follows that \(Q \cap P\) is quasiconvex in \(Q\) and hence hyperbolic \cite[Proposition III.\(\Gamma\).3.7]{Bridson_Haefliger}.
    On the other hand, if \(Q \cap P\) is finite then it is trivially hyperbolic.
\end{proof}

Mart\'{i}nez-Pedroza and Sisto proved the following combination theorem for relatively quasiconvex subgroups with compatible parabolics.

\begin{theorem}[{\cite[Theorem 2]{MPS}}]
\label{thm:compat_parab_amalgam}
    Let \(Q\) and \(R\) be relatively quasiconvex subgroups with compatible parabolics, and let \(S' \leqslant_f S = Q \cap R\) be a finite index subgroup of their intersection.
    There is a constant \(M = M(Q,R,S') \geq 0\) such that the following is true.

    If \(Q' \leqslant Q\) and \(R' \leqslant R\) satisfy \(Q' \cap R' = S'\) and \(\abs{g}_X \geq M\) for all \(g \in (Q' \cup R') \setminus S'\), then \(\langle Q', R' \rangle\) is relatively quasiconvex and \(\langle Q', R' \rangle \cong Q' \ast_{S'} R'\).
\end{theorem}

\subsection{Profinite topology}
\label{subsec:profinite_topology}

Any group \(G\) can be equipped with the \emph{profinite topology}, which is based by left cosets of finite index subgroups of \(G\).
A subset of \(G\) is said to be \emph{separable} if it is closed in the profinite topology on \(G\).
A group \(G\) is said to be \emph{residually finite} if the trivial subgroup is separable, \emph{LERF} if every finitely generated subgroup of \(G\) is separable, and \emph{double coset separable} if every product of two finitely generated subgroups of \(G\) is separable.
Each of these properties pass to subgroups and finite index supergroups.
As an example, polycyclic groups are known to be double coset separable \cite{L-W}.

A relatively hyperbolic group is called \emph{QCERF} if each of its finitely generated relatively quasiconvex subgroups is separable.
Whether all groups hyperbolic relative to a collection of LERF and \emph{slender} (i.e. every subgroup is finitely generated) subgroups are QCERF is equivalent to a well-known open problem \cite[Theorem 1.2]{MMPSep}.
Many common examples of relatively hyperbolic groups are known to be QCERF, for example limit groups \cite{WiltonLimitGps}, \(C'(1/6)\) small cancellation quotients of LERF groups \cite[Theorem 1.7]{Einstein-Ng,MMPSep}, and geometrically finite Kleinian groups \cite{Agol}.

Our use of the profinite topology in this paper goes through the following elementary fact.

\begin{lemma}
\label{lem:sep_from_finite_set}
    Let \(G\) be a group, \(H \leqslant G\) a separable subgroup, and \(U \subseteq G\) a finite subset of \(G\) with \(H \cap U = \emptyset\).
    Then there is a finite index subgroup \(G' \leqslant_f G\) with \(H \subseteq G'\) and \(G' \cap U = \emptyset\).
    Moreover, if \(H\) is normal, \(G'\) may be taken to be a finite index normal subgroup of \(G\).
\end{lemma}

\begin{proof}
    Write \(U = \{u_1, \dots, u_n\}\). 
    As \(H\) is closed in the profinite topology, it is the intersection of the finite index subgroups containing it.
    Thus, for each \(i = 1, \dots, n\) there is \(G_i \leqslant_f G\) with \(H \subseteq G_i\) and \(u_i \notin G_i\).
    Then \(G' = \bigcap_{i=1}^n G_i \leqslant_f G\) satisfies the lemma statement.
    When \(H\) is normal is \(G\), we may replace each \(G_i\) with its normal core to obtain the latter statement. 
\end{proof}


\section{Controlling parabolic subgroups in relatively hyperbolic groups}
\label{sec:parab}

For this section, we let \(G\) be hyperbolic relative to \(\relto\) with finite relative generating set \(X\), and let \(L \geq 0\) be the constant and \(\Omega \subseteq G\) be the finite set provided by Proposition~\ref{prop:osin_polygon}.

\begin{proposition}
\label{prop:int_of_parabs_conj_to_short}
    Let \(a, b \in G\) and \(\lambda, \nu \in \Nu\) be such that \(a H_\lambda a^{-1} \ne b H_\nu b^{-1}\).
    Then each element of \(a H_\lambda a^{-1} \cap b H_\nu b^{-1}\) is conjugate to an element \(h \in G\) with \(\abs{h}_\Omega \leq 4L\).
\end{proposition}

\begin{proof}
    Conjugating if necessary, we may assume that \(a = 1\). 
    Further, suppose that \(b \in G\) is such that \(\abs{b}_{X \cup \mathcal{H}}\) is minimal among elements in the coset \(b H_\nu\).
    Now let \(g \in H_\lambda \cap b H_\nu b^{-1}\) be a nontrivial element, and let \(h \in H_\nu\) be such that \(g = bhb^{-1}\).

    Let \(p\) be a geodesic in \(\relcay\) with \(p_- = 1\) and \(p_+ = b\). 
    Further, let \(u\) be the \(H_\lambda\)-edge of \(\relcay\) with \(u_- = 1\) and \(\elem{u} = g\), and let \(v\) be the \(H_\nu\)-edge of \(\relcay\) with \(v_- = b\) and \(\elem{v} = h\).
    Note that \(v_+ = b h = g b\) by definition, so that \(p' = g \cdot p\) (i.e. the translate of \(p\) by \(g\)) has endpoints \(u_+\) and \(v_+\).
    Now consider the geodesic quadrilateral \(Q\) with sides \(u\), \(p\), \(v\), and \(p'\).
    We will show that \(v\) is isolated in \(Q\).

    Suppose otherwise, for a contradiction.
    If \(u\) and \(v\) are connected, then we must have \(\lambda = \nu\) and both \(u_- = 1\) and \(v_- = b\) lie in the same \(H_\lambda\)-coset.
    However, this means that \(H_\lambda = b H_\nu b^{-1}\), contrary to the assumption.
    Therefore \(v\) must be connected to an \(H_\nu\)-component \(s\) of either \(p\) or \(p'\).
    We suppose, without loss of generality, that \(s\) lies in \(p\).
    Since \(v\) and \(s\) are connected and \(p_+ = v_- = b\), the endpoints of \(s\) satisfy \(d_{X\cup\mathcal{H}}(s_-, p_+) \leq 1\) and \(d_{X\cup\mathcal{H}}(s_+, p_+) \leq 1\).
    Therefore \(s\) must be the terminal edge of \(p\), for otherwise the geodesicty of \(p\) is contradicted.
    That is to say, \(s\) is an \(H_\nu\)-component of \(p\) with \(s_+ = p_+\).
    But then \(b \elem{s}^{-1} \in b H_\nu\) and 
    \[\abs{b \elem{s}^{-1}}_{X\cup\mathcal{H}} = d_{X \cup \mathcal{H}}(1,p_+ s_+^{-1} s_- ) = d_{X\cup\mathcal{H}}(1,s_-) < \abs{p}_{X\cup\mathcal{H}} = \abs{b}_{X\cup\mathcal{H}}\]
    contradicting the minimality of \(b\).

    As \(s\) cannot contain \(u\) or be an \(H_\lambda\)-component of \(p\) or \(p'\), \(v\) is isolated in \(Q\).
    Proposition~\ref{prop:osin_polygon} then tells us that \(\abs{h}_\Omega = \abs{v}_\Omega \leq 4L\), as required.
\end{proof}

As the set \(\Omega\) is finite, there are only finitely many elements of \(G\) whose length with respect to \(\Omega\) is less than any given number.
The following is then immediate.

\begin{corollary}
\label{cor:finitely_many_conj_classes_of_doubly_parabolic}
    There are finitely many conjugacy classes of elements in \(G\) belonging to more than one maximal parabolic subgroup.
\end{corollary}

We can use the above to control the intersections of relatively quasiconvex subgroups with maximal parabolic subgroups in a residually finite hyperbolic group.

\begin{proposition}
\label{prop:killing_finite_parabolics}
    Suppose that \(G\) is residually finite, and let \(Q \leqslant G\) be a relatively quasiconvex subgroup.
    Then there is a finite index subgroup \(Q' \leqslant_f Q\) such that for any maximal parabolic subgroup \(P \leqslant G\), the subgroup \(Q' \cap P\) is either infinite or trivial.
\end{proposition}

\begin{proof}
    Let \(S\) be a set of representatives of conjugacy classes of nontrivial elements of \(G\) belonging to more than one maximal parabolic subgroups of \(G\).
    Corollary~\ref{cor:finitely_many_conj_classes_of_doubly_parabolic} tells us that the set \(S\) is finite.
    That \(G\) is residually finite means exactly that the trivial subgroup \(\{1\}\) is separable, and \(\{1\} \cap S = \emptyset\) so Lemma~\ref{lem:sep_from_finite_set} gives us a finite index normal subgroup \(G_1\lhd_f G\) with \(G_1 \cap S = \emptyset\).
    As \(G_1\) is normal, it thus contains no nontrivial elements that belong to more than one maximal parabolic subgroup of \(G\).

    Let \(Q_1 = G_1 \cap Q \lhd_f Q\).
    Now by \cite[Theorem 4.2]{OsinRHG}, there are only finitely many conjugacy classes of finite order hyperbolic elements in \(Q_1\) (an element of \(G\) is called \emph{hyperbolic} if it is not conjugate to an element of \(H_\nu\) for any \(\nu \in \Nu\)).
    Similarly to before, by residual finiteness there is \(Q' \lhd_f Q_1\) excluding each of these elements by Lemma~\ref{lem:sep_from_finite_set}.
    
    We will show that the subgroup \(Q' \leqslant_f Q\) has the desired property.
    Let \(\mathcal{P}\) be a set of maximal parabolic subgroups of \(G\) such that \(Q_1\) is hyperbolic relative to the collection of infinite subgroups \(\{Q' \cap H \, | \, H \in \mathcal{P}\}\) (see \cite[Theorem 9.4]{HruskaRHCG}).
    Let \(P \leqslant G\) be a maximal parabolic subgroup of \(G\), and suppose that \(Q' \cap P\) is nontrivial.
    If \(Q' \cap P\) contains an element of infinite order then we are done, so suppose \(x \in Q' \cap P\) is a nontrivial element of finite order.
    By construction, \(Q'\) contains no elements of finite order that are hyperbolic in \(Q_1\), so \(x\) must be parabolic in \(Q_1\).
    That is, there is \(q \in Q_1\) such that \(qxq^{-1} \in Q_1 \cap H\) for some \(H \in \mathcal{P}\).
    It follows that \(x \in Q_1 \cap P \cap q^{-1} H q \subseteq G_1 \cap P \cap q^{-1} H q\), whence we must have \(P = q^{-1} H q\) by the definition of \(G_1\).
    This implies that
    \[
        Q_1 \cap P = Q_1 \cap q^{-1} H q = q^{-1}(Q_1 \cap H) q
    \]
    and since \(Q_1 \cap H\) is infinite, \(Q_1\cap P\) is infinite as well.
    The result then follows by noting that \(Q' \cap P\) has finite index in \(Q_1 \cap P\).
\end{proof}


\section{Quasigeodesics and shortcuttings}
\label{sec:shortcutting}

\begin{convention}
\label{conv:GQR}
    For the remainder of this paper, we will use the convention that \(G\) is a group with finite generating set \(X\) and \(G\) is hyperbolic relative to the subgroups \(\relto\).
    \(Q\) and \(R\) will be finitely generated relatively quasiconvex subgroups of \(G\), and we will denote \(S = Q \cap R\).
    Moreover, \(\delta \geq 0\) will be a hyperbolicity constant for \(\relcay\) and \(\varepsilon \geq 0\) will be a quasiconvexity constant for both \(Q\) and \(R\).
\end{convention}

In this section we recall the construction of the shortcutting of a broken line in \(\relcay\) from \cite[Section 9]{MinMin} and show that shortcuttings of broken lines satisfying certain metric conditions have nice properties.
Analysing shortcuttings of broken lines comprises the main technical tool that we use to understand elements of joins of subgroups of \(G\). 

\begin{procedure}[$\Theta$-shortcutting]
\label{proc:shortcutting}
    Fix a natural number \(\Theta \in \NN\) and let \(p = p_1 \dots p_n\) be a broken line in \(\Gamma(G,X\cup\mathcal{H})\). Let \(v_0, \dots, v_d\) be the enumeration of all vertices of \(p\) in the order they occur along the path (possibly with repetition), so that \(v_0 = p_-\), \(v_d = p_+\) and $d=\ell(p)$.

    We construct broken lines \(\Sigma(p,\Theta)\) and \(\Sigma_0(p,\Theta)\), which we call \(\Theta\)-\emph{shortcuttings} of \(p\), which come with a finite set  \(V(p,\Theta) \subset \{0,\dots,d\} \times \{0,\dots,d\}\) corresponding to indices of vertices of \(p\) that we shortcut along.

    In the algorithm below we will refer to numbers \(s,t,N \in \{0,\dots,d\}\) and a subset \(V \subseteq \{0,\dots,d\} \times \{0,\dots,d\}\). To avoid excessive indexing these will change value throughout the procedure.
    The parameters $s$ and $t$ will indicate the starting and terminal vertices of subpaths of $p$ in which all $\mathcal H$-components have lengths less than $\Theta$. The parameter $N$ will keep track of how far along the path $p$ we have proceeded. The set $V$ will collect all pairs of indices $(s,t)$ obtained during the procedure.
    We initially take \(s = 0\), $N=0$ and \(V = \emptyset\).
    
    \begin{steps}
        \item
            If there are no edges of \(p\) between \(v_N\) and \(v_d\) that are labelled by elements of \(\mathcal{H}\), then add the pair $(s,d)$ to the set $V$ and skip ahead to Step 4.
            Otherwise, continue to Step 2.
        \item
            Let \(t \in \{0,\dots,d\}\) be the least natural number with \(t \geq N\) for which the edge of \(p\) with endpoints \(v_t\) and \(v_{t+1}\) is an \(\mathcal{H}\)-component $h_i$ of a geodesic segment $p_i$ of \(p\), for some \(i \in \{1, \dots, n\}\).

            If $i=n$ or if $h_i$ is not connected to a component of $p_{i+1}$ then set $j=i$. Otherwise, let \(j \in \{i+1,\dots,n\}\) be the maximal integer such that \(p\) has consecutive backtracking along \(\mathcal{H}\)-components \(h_i, \dots, h_j\) of segments \(p_i, \dots, p_j\).             
            Proceed to Step 3.

        \item 
            If \[\max\Big\{\abs{h_k}_X \, \Big| \, k = i, \dots, j\Big\} \geq \Theta,\] then add the pair \((s,t)\) to the set \(V\) and redefine $s=N$   in \(\{1,\dots,d\}\) to be the index of the vertex $(h_j)_+$ in the above enumeration $v_0,\dots,v_d$ of the vertices of $p$.
            Otherwise let $N$ be the index of $(h_i)_+$, and leave $s$ and \(V\) unchanged.
            
            Return to Step~1 with the new values of \(s\), $N$ and \(V\).
        \item
            Set \(V(p,\Theta) = V\). The above constructions gives a natural ordering of $V(p,\Theta)$: \[V(p,\Theta) = \{(s_0, t_0), \dots, (s_m,t_m)\},\] where \(s_k  \le t_k <  s_{k+1}\), for all \(k = 0, \dots, m-1\). Note that $s_0=0$ and $t_m=d$. Proceed to Step 5.
        \item
            For each $k=0,\dots,m$, let $f_k$ be a geodesic segment (possibly trivial) connecting $v_{s_k}$ with $v_{t_k}$. 
            Similarly for each \(k = 0, \dots, m\) let \(p'_k\) be the (possibly trivial) subpath of \(p\) with endpoints \(v_{s_k}\) and \(v_{t_k}\). 
            Note that when $k <m$, \(v_{t_k}\) and \(v_{s_{k+1}}\) are in the same left coset of $H_\nu$, for some $\nu \in \Nu$. 
            If \(v_{t_k} = v_{s_{k+1}}\) then let \(e_k\) be the trivial path at \(v_{t_k}\), otherwise let \(e_k\) be an edge of $\relcay$ starting at \(v_{t_k}\), ending at \(v_{s_{k+1}}\) and labelled by an element of \(H_\nu \setminus\{1\}\).

            We define the broken line \(\Sigma(p,\Theta)\) to be the concatenation \(f_0 e_1 f_1 e_2 \dots f_{m-1} e_m f_m\).
            We also define the broken line \(\Sigma_0(p,\Theta)\) to be the concatenation \(p'_0 e_1 p'_1 e_2 \dots p'_{m-1} e_m p'_m\).
    \end{steps}
\end{procedure}

\begin{remark}
\label{rem:cpts_of_p'k}
    Suppose that \(p\) is a broken line in \(\relcay\), \(\Theta \in \NN\), and let \(\Sigma_0(p,\Theta) = p'_0 e_1 p'_1 e_2 \dots p'_{m-1} e_m p'_m\) be the shortcutting obtained from Procedure~\ref{proc:shortcutting}.
    For each \(i = 0, \dots m\), the subpath \(p'_i\) is a broken line, each of whose segments contain no \(\mathcal{H}\)-components \(h\) with \(\abs{h}_X \geq \Theta\).
    In particular, if \(\Theta = 1\) then the paths \(p'_0, \dots, p'_m\) contain no edges labelled by elements of \(\mathcal{H}\).
\end{remark}

We recall also the definition of tamable broken lines, which serve as the prototypical input for Procedure~\ref{proc:shortcutting}.

\begin{definition}[Tamable broken line]
\label{def:tamable}
    Let \(p = p_1 \dots p_n\) be a broken line in $\relcay$, and let \(B, C, \zeta \geq 0, \Theta \in \NN\).
    We say that \(p\) is \emph{\((B,C,\zeta,\Theta)\)-tamable} if all of the following conditions hold:
    \begin{enumerate}[label=(\roman*)]
        \item \label{cond:tam_1} \(\abs{p_i}_X \geq B\), for \(i = 2, \dots, n-1\);
        \item \label{cond:tam_2} \(\langle (p_i)_-, (p_{i+1})_+ \rangle_{(p_i)_+} \leq C\), for each \(i = 1, \dots, n-1\);
        \item \label{cond:tam_3} whenever \(p\) has consecutive backtracking along \(\mathcal{H}\)-components \(h_i, \dots, h_j\), of segments \(p_i, \dots, p_j\), such that
            \[
                \max\Big\{ \abs{h_k}_X \, \Big| \, k = i, \dots, j \Big\} \geq \Theta,
            \]
        it must be that \(d_X\Bigl( (h_i)_-,(h_j)_+ \Bigr) \geq \zeta\).
    \end{enumerate}
\end{definition}

One of the central technical results obtained in \cite{MinMin} is the following.

\begin{proposition}[{\cite[Proposition 9.4]{MinMin}}]
\label{prop:shortcutting_quasigeodesic}
    Given arbitrary \(C \geq 14\delta\) and \(\eta \geq 0\) there are constants \( \lambda = \lambda(C) \geq 1\), \(c = c(C) \geq 0\) and \(\zeta = \zeta(\eta,C) \geq 1\) such that for any natural number \(\Theta \geq \zeta\) there is \(B_0 = B_0(\Theta,C) \geq 0\) satisfying the following.

    Let \(p = p_1 \dots p_n\) be a \((B_0,C,\zeta,\Theta)\)-tamable broken line in $\relcay$ and let \(\Sigma(p,\Theta)\) be the \(\Theta\)-shortcutting, obtained by applying Procedure~\ref{proc:shortcutting} to \(p\), with \(\Sigma(p,\Theta) = f_0 e_1 f_1 \dots f_{m-1} e_m f_m\). 
    Then $e_k$ is non-trivial for each $k=1,\dots,m$ and $\Sigma(p,\Theta)$  is \((\lambda,c)\)-quasigeodesic without backtracking.  

    Moreover, for any \(k \in \{ 1, \dots, m\}\), if we denote by \(e'_k\) the \(\mathcal{H}\)-component of \(\Sigma(p,\Theta)\) containing \(e_k\), then  \(\abs{e'_k}_X \geq \eta\).
\end{proposition}

As part of the proof of the above, one obtains the following under the same hypotheses:

\begin{lemma}[{\cite[Lemma 9.7]{MinMin}}]
\label{lem:fk_quasigeodesic}
    There is a constant \(c_0 = c_0(C)\) such that the subpaths of \(p\) between \(v_{s_k}\) and \(v_{t_k}\) are \((4,c_0)\)-quasigeodesic for each \(k = 0, \dots, m\).
\end{lemma}

\begin{lemma}[{\cite[Lemma 9.8]{MinMin}}]
\label{lem:short_ending_components}
    There is a constant \(\rho = \rho(C) \geq 0\) such that if \(k \in \{0, \dots, m-1\}\) and \(h\) is an \(\mathcal{H}\)-component of \(f_{k}\) or \(f_{k+1}\) that is connected to \(e_{k+1}\), then \(\abs{h}_X \leq \rho\).
\end{lemma}

We begin with the following observation.

\begin{lemma}
\label{lem:shortuctting_of_quasigeodesic}
    Let \(\lambda \geq 1\) and \(c \geq 0\).
    Let \(p = p_1 \dots p_n\) be a \((\lambda,c)\)-quasigeodesic broken line in \(\relcay\) with \(\abs{p_i}_{X\cup\mathcal{H}} > \lambda + c\) for each \(i = 1, \dots, n\).
    Then the path \(\Sigma_0(p,1)\) obtained from Procedure~\ref{proc:shortcutting} is a \((\lambda,c)\)-quasigeodesic without backtracking.
\end{lemma}

\begin{proof}
    Let \(q\) be a subpath of \(\Sigma_0 = \Sigma_0(p,1) = p'_0 e_1 p'_1 \dots p'_{m-1} e_m p'_m\).
    Since for each \(i\), \(p'_i\) is a subpath of \(p\) and \(e_i\) consists of at most a single edge, \(q_-\) and \(q_+\) are vertices of \(p\).
    Let \(p'\) be the subpath of \(p\) with \(p'_- = q_-\) and \(p'_+ = q_+\).
    The path \(q\) can be obtained by replacing subpaths of \(p'\) with single edges, so that the length of \(q\) is bounded by the length of \(p'\).
    Then by the quasigeodescity of \(p\) we have
    \[
        \ell(q) \leq \ell(p') \leq \lambda d_{X\cup\mathcal{H}}(p'_-,p'_+) + c = \lambda d_{X\cup\mathcal{H}}(q_-,q_+) + c,
    \]
    so \(\Sigma_0\) is \((\lambda,c)\)-quasigeodesic.
    
    We must now show that \(\Sigma_0\) is without backtracking, so suppose for a contradiction that it does have backtracking.
    As noted in Remark~\ref{rem:cpts_of_p'k} the subpaths \(p'_0, \dots, p'_m\) contain no \(\mathcal{H}\)-subpaths.
    That is, if \(h\) is an \(\mathcal{H}\)-subpath of \(\Sigma_0\), it must be one of the paths \(e_1, \dots, e_m\).
    Therefore it must be that there are integers \(1 \leq k < l \leq m\) such that \(e_k\) and \(e_l\) are nontrivial connected \(\mathcal{H}\)-subpaths of \(\Sigma_0\).
    Thus,
    \begin{equation}
    \label{eq:hi_hj_conn}
        d_{X\cup\mathcal{H}}((e_k)_+,(e_l)_-) \leq 1.
    \end{equation}

    Let \(h_1\) be the \(\mathcal{H}\)-component of a segment of \(p\) with \((h_1)_+ = (e_k)_+\) and let \(h_2\) be the \(\mathcal{H}\)-component of a segment of \(p\) with \((h_2)_- = (e_l)_-\).
    Since \(e_k\) and \(e_l\) are connected, so are \(h_1\) and \(h_2\).
    Following Remark~\ref{rem:comp_of_geod_is_an_edge}, \(h_1\) and \(h_2\) cannot lie in the same segment of \(p\).
    If \(h_1\) and \(h_2\) lie in adjacent segments of \(p\), then they are part of the same instance of consecutive backtracking and the construction of \(\Sigma_0\) is contradicted.
    Otherwise, the path \(p'_k\) contains a full segment, say \(p_{s}\), of \(p\).
    Then by quasigeodesicity of \(p\) and (\ref{eq:hi_hj_conn}),
    \begin{equation}
    \label{eq:bound_on_ps_length}
        \ell(p_s) \leq \ell(p'_k) \leq \lambda d_{X\cup\mathcal{H}}((e_k)_+,(e_l)_-) + c \leq \lambda + c.
    \end{equation}
    However, since \(p_s\) is a geodesic, \(\abs{p_s}_{X\cup\mathcal{H}} = \ell(p_s)\).
    Therefore (\ref{eq:bound_on_ps_length}) contradicts the lemma hypothesis that \(\ell(p_s) > \lambda + c\).
\end{proof}

For the remainder of the section, we fix as constants some \(C \geq 14\delta\) and \(\eta \geq 0\), let \(\lambda = \lambda(C), c = c(C)\), and \(\zeta = \zeta(C,\eta)\) be the constants obtained from Proposition~\ref{prop:shortcutting_quasigeodesic}, and let \(c_0 = c_0(C)\) be the constant of Lemma~\ref{lem:fk_quasigeodesic}.
Let \(L \geq 0\) be the constant of Proposition~\ref{prop:osin_polygon}.

\begin{lemma}
\label{lem:backtracking_between_fk}
    For any \(\Theta \geq \zeta\) there is \(E_0 = E_0(\Theta) \geq 0\) such that for any \(B \geq E_0\) the following is true.
    
    Let \(p = p_1 \dots p_n\) be a \((B,C,\zeta,\Theta)\)-tamable broken line and let \(V(p,\Theta) = \{(s_k,t_k) \mid k = 0, \dots, m\}\) be the set from Procedure~\ref{proc:shortcutting}.
    Fix \(k \in \{0, \dots, m\}\) and denote by \(p'\) the subpath of \(p\) with \(p'_- = v_{s_k}\) and \(p'_+ = v_{t_k}\).
    If \(q\) and \(r\) are connected \(\mathcal{H}\)-edges of \(p'\), then \(d_X(q_-,r_+) \leq 3L + 2\Theta\).
\end{lemma}

\begin{proof}
    We begin by choosing the constant
    \[
        E_0 = \max\{B_0, 4(1+c_0)\Theta\} + 1,
    \]
    where \(B_0 = B_0(\Theta,C)\) is obtained from Proposition~\ref{prop:shortcutting_quasigeodesic}, and let \(B \geq E_0\).
    The path \(p'\) is a broken line with \(p' = p'_i p_{i+1} \dots p_{j-1} p'_j\), where \(p'_i\) (respectively, \(p'_j\)) is a subpath of \(p_i\) with \((p'_i)_- = v_{s_k}\) and \((p'_i)_+ = (p_i)_+\) (respectively, of \(p_j\) with \((p'_j)_- = (p_j)_-\) and \((p'_j)_+ = v_{t_k}\)).
    As in Remark~\ref{rem:cpts_of_p'k}, each \(\mathcal{H}\)-component \(h\) of the paths \(p'_i, p_{i+1}, \dots, p_{j-1}, p'_j\) satisfies \(\abs{h}_X \leq \Theta\).
    This implies
    \begin{equation}
    \label{eq:len_of_q_r}
        \abs{q}_X + \abs{r}_X \leq 2\Theta.
    \end{equation}
    Since each segment of \(p\) is geodesic,  \(q\) and \(r\) must be connected \(\mathcal{H}\)-components of distinct segments of \(p\), say \(p_k\) and \(p_l\).
    Without loss of generality, we assume \(k < l\).
    If \(l > k+1\) then the subpath of \(p'\) between \(q_-\) and \(r_+\) contains the entire segment \(p_{k+1}\).
    By Lemma~\ref{lem:rel_paths_with_short_comps},
    \begin{equation}
    \label{eq:rel_len_of_pk+1}
        \ell(p_{k+1}) \geq \frac{1}{\Theta}\abs{p_{k+1}} \geq \frac{B}{\Theta},
    \end{equation}
    where the last inequality is given by condition \ref{cond:tam_1} of tamability.
    
    Lemma~\ref{lem:fk_quasigeodesic} tells us that \(p'\) is \((4,c_0)\)-quasigeodesic.
    Combining this fact with (\ref{eq:rel_len_of_pk+1}) and the choice of \(B\), we have
    \[
        d_{X\cup\mathcal{H}}(q_-,r_+) \geq \frac{1}{4}\ell(p_{k+1}) - \frac{c_0}{4} \geq \frac{B - \Theta c_0}{4\Theta} > 1.
    \]
    On the other hand, \(q\) and \(r\) are connected, so that \(d_{X\cup\mathcal{H}}(q_-,r_+) \leq 1\), a contradiction.
    Therefore \(q\) and \(r\) must lie in adjacent segments \(p_k\) and \(p_{k+1}\) of \(p\).
    
    If \(q_+ \ne r_-\), then there is an \(\mathcal{H}\)-edge \(h\) in \(\relcay\) with \(h_- = q+\) and \(h_+ = r_-\).
    The edge \(h\) must be isolated in the triangle \(h \cup [q_+,(p_k)_+] \cup [(p_k)_+,r_-]\).
    Thus by Proposition~\ref{prop:osin_polygon}, we have \(\abs{h}_X = d_X(q_+,r_-) \leq 3L\).
    Otherwise \(q_+ = r_-\) and \(d_X(q_+,r_-) = 0\).
    Together with (\ref{eq:len_of_q_r}), we obtain
    \begin{align*}
        d_X(q_-,r_+) &\leq d_X(q_-,q_+) + d_X(q_+,r_-) + d_X(r_-,r_+) \\
            &\leq \abs{q}_X + 3L + \abs{r}_X \leq 3L + 2\Theta,
    \end{align*}
    as required.
\end{proof}

The following combines the results of this section in a format we will find useful.

\begin{lemma}
\label{lem:double_shortcutting}
    For each \(\Theta \geq \zeta\), there is \(E_1 = E_1(\Theta) \geq 0\) such that for any \(B \geq E_1\) the following holds.
    
    Let \(p = p_1 \dots p_n\) be a \((B,C,\zeta,\Theta)\)-tamable broken line and denote by \(\Sigma_0(p,\Theta)\) the broken line \(p'_0 e_1 p'_1 \dots p'_{m-1} e_m p'_m\) obtained from Procedure~\ref{proc:shortcutting}.
    Then for each \(i = 0, \dots, m\), the shortcutting \(\Sigma_0(p'_i,1)\) is a \((4,c_0)\)-quasigeodesic without backtracking, and each of its \(\mathcal{H}\)-components \(h\) satisfies \(\abs{h}_X \leq 3L + 2\Theta\).
\end{lemma}

\begin{proof}
    Define
    \[
        E_1 = \max\{E_0(\Theta), (4+c_0)\Theta\} + 1,
    \]
    where \(E_0(\Theta)\) is the constant obtained from Lemma~\ref{lem:backtracking_between_fk}, and let \(B \geq E_1\).
    Lemma~\ref{lem:fk_quasigeodesic} tells us that for each \(0 \leq i \leq m\), the path \(p'_i\) is a \((4,c_0)\)-quasigeodesic.
    Let \(t\) be a segment of \(p'_i\), which is a geodesic.
    By construction, any \(\mathcal{H}\)-component \(h\) of \(t\) has \(\abs{h}_X \leq \Theta\), so by Lemma~\ref{lem:rel_paths_with_short_comps}
    \[
        \abs{t}_{X\cup\mathcal{H}} = \ell(t) \geq \frac{B}{\Theta} \geq \frac{E_1}{\Theta} > 4 + c_0.
    \]
    Hence by Lemma~\ref{lem:shortuctting_of_quasigeodesic}, \(\Sigma_0(p'_i,1)\) is a \((4,c_0)\)-quasigeodesic without backtracking.
    Each of the \(\mathcal{H}\)-components \(h\) of \(\Sigma_0(p'_i,1)\) is either an \(\mathcal{H}\)-component of a segment of \(p'_i\) or shares its endpoints with two connected \(\mathcal{H}\)-components of segments of \(p'_i\).
    Therefore by Lemma~\ref{lem:backtracking_between_fk}, \(\abs{h}_X \leq 3L + 2\Theta\).
\end{proof}


\section{Shortcuttings for parabolic paths}
\label{sec:parab_paths}

In this section we study the behaviour of shortcuttings of tamable broken lines that represent elements from parabolic subgroups of \(G\).
The aim is to show that tamable broken lines representing elements of some \(b H_\nu b^{-1}\) consist of essentially a single instance of consecutive backtracking that involves all its segments, given that the element is sufficiently long in comparison to the conjugator \(b\).

As a simplifying assumption, throughout this section we will assume that \(b\) is such that \(\abs{b}_{X\cup\mathcal{H}}\) is minimal among elements in its left \(H_\nu\)-coset.
We observe that it does not cost us a lot to make such an assumption.

\begin{remark}
\label{rem:no_ending_cpts}
    Let \(b \in G\) and \(\nu \in \Nu\).
    Suppose \(\abs{b}_{X\cup\mathcal{H}}\) is not minimal among elements of \(bH_\nu\).
    Let \(b_1 = bh \in bH_\nu\) be such a minimal element, so that \(\abs{b_1}_{X\cup\mathcal{H}} < \abs{b}_{X \cup \mathcal{H}}\).
    Since \(b_1 \in bH_\nu\), it must be that \(\abs{b}_{X\cup\mathcal{H}} \leq \abs{b_1}_{X \cup \mathcal{H}} + 1\).
    Combining these inequalities, we in fact have that \(\abs{b}_{X\cup\mathcal{H}} = \abs{b_1}_{X \cup \mathcal{H}} + 1\).
    Therefore the path \(p = [1,b_1] \cup e\), where \(e\) is a \(H_\nu\)-edge labelled by \(h^{-1}\), is a geodesic in \(\relcay\).
    Moreover, if \(\abs{b}_X \leq M\) then by Lemma~\ref{lem:len_of_subgeodesic}, \(\abs{b_1}_X = d_X(1,e_-) \leq \xi_0 M^2\) where \(\xi_0\) is the constant of that lemma.
\end{remark}

\begin{lemma}
\label{lem:comp_of_parab_close_to_geodesic}
    For any \(M \geq 0\) there is \(N_0 = N_0(M) \geq 1\) such that the following is true.
    
    Let \(b \in G\) with \(\abs{b}_X \leq M\), and let \(p\) be a geodesic in \(\relcay\) with \(\elem{p} \in bH_\nu b^{-1}\) for some \(\nu \in \Nu\). 
    Suppose that \(\abs{b}_{X\cup\mathcal{H}}\) is minimal among elements of \(bH_\nu\) and denote by \(h\) the \(H_\nu\)-edge with \(h_- = p_- b\) and \(\elem{p} = b \elem{h} b^{-1}\).
    If \(\abs{p}_X \geq N_0\), then \(h\) is connected to an \(H_\nu\)-component \(h'\) of \(p\) with 
    \[
        d_X(h_-,h'_-) \leq 3L \quad \textrm{ and } \quad d_X(h_+,h'_+) \leq 3L,
    \]
    where \(L\) is the constant from Proposition~\ref{prop:osin_polygon}.
\end{lemma}

\begin{proof}
    Take \(N_0 = 2M + \kappa\), where \(\kappa = \kappa(1,0,M)\) is the constant from Proposition~\ref{prop:bcp} applied to \(M\)-similar geodesics.
    Suppose that \(\abs{p}_X \geq N_0\), so that \(\abs{h}_X \geq \abs{p}_X - 2M \geq \kappa\) by the triangle inequality.
    Now we apply Proposition~\ref{prop:bcp} to the \(M\)-similar geodesics \(h\) and \(p\), which shows that \(h\) is connected to an \(H_\nu\)-component \(h'\) of \(p\).
    If \(h'_- = h_-\), then we are done, so suppose otherwise.
    Take \(s = [h_-,h'_-]\) to be the \(H_\nu\)-edge in \(\relcay\) labelled by the element \(h_-^{-1}h'_- \in H_\nu\).
    
    We will show that \(s\) is isolated in the geodesic triangle \([p_-,h_-] \cup [p_-,h'_-] \cup s\), whence we can conclude that \(\abs{s}_X = d_X(h_-,h'_-) \leq 3L\) by applying Proposition~\ref{prop:osin_polygon}.
    Suppose for a contradiction that \(s\) is connected to an \(H_\nu\)-component \(t\) of \([p_-,h_-]\).
    Since \(s\) is connected to \(h\), \(h\) is also connected to \(t\).
    That is, the vertices \(t_-\) and \(h_-\) lie in the same \(H_\nu\)-coset which implies that \(d_{X\cup\mathcal{H}}(t_-,h_-) \leq 1\).
    However, by minimality of \(\abs{b}_{X\cup\mathcal{H}}\) among elements of \(bH_\nu\), \(t\) cannot be the final edge of \([p_-,h_-]\).
    This means that \(d_{X\cup\mathcal{H}}(t_-,h_-) \geq 2\) by geodesicity of \([p_-,h_-]\), a contradiction.
    Similarly, if \(s\) is connected to an \(H_\nu\)-component \(t\) of \([p_-,h'_-]\), then \(t\) is in turn connected to \(h'\), this time contradicting geodesicity of \(p\) (via Remark~\ref{rem:comp_of_geod_is_an_edge}).
    
    Thus \(d_X(h_-,h'_-) \leq 3L\) by Proposition~\ref{prop:osin_polygon}.
    The same argument, by symmetry, shows that \(d_X(h_+,h'_+) \leq 3L\).
\end{proof}
 
For the remainder of the section, we fix a constant \(C \geq 14\delta\), let \(\lambda = \lambda(C)\) and \(c = c(C)\) be the constants obtained from Proposition~\ref{prop:shortcutting_quasigeodesic}, and let \(c_0 = c_0(C)\) be the constant of Lemma~\ref{lem:fk_quasigeodesic}.
Given any \(\eta \geq 0\) we will write \(\zeta(\eta,C)\) for the constant obtained from the same proposition.
Finally, given any \(\Theta \geq \zeta(\eta,C)\), we write \(E_1(\Theta)\) for the constant obtained from Lemma~\ref{lem:double_shortcutting}.

\begin{lemma}
\label{lem:h_conn_to_ek}
    There is a constant \(\kappa_0 = \kappa_0(C) \geq 0\) such that for any \(M \geq 0, \eta \geq 0, \Theta \geq \zeta = \zeta(\eta,C), B \geq E_1(\Theta)\) there is \(N_1 = N_1(\Theta,M) \geq 1\) such that the following holds.
    
    Let \(b \in G\) with \(\abs{b}_X \leq M\).
    Let \(p = p_1 \dots p_n\) be a \((B,C,\zeta,\Theta)\)-tamable broken line, and suppose that \(\elem{p} \in b H_\nu b^{-1}\) for some \(\nu \in \Nu\).
    Suppose that \(\abs{b}_{X\cup\mathcal{H}}\) is minimal among elements of \(bH_\nu\) and denote by \(\Sigma(p,\Theta) = f_0 e_1 f_1 \dots f_{m-1} e_m f_m\) its \(\Theta\)-shortcutting.
    Let \(h\) be the \(H_\nu\)-edge in \(\relcay\) with \(h_- = p_-b\) such that \(\elem{p} = b\elem{h}b^{-1}\).
    
    If \(\abs{p}_X \geq N_1\), then \(h\) is connected to \(e_k\) for some \(k = 1, \dots, m\) and
    \[
        d_X(h_-,(e_k)_-) \leq \kappa_0 \quad \textrm{ and } \quad d_X(h_+,(e_k)_+) \leq \kappa_0.
    \]
\end{lemma}

\begin{figure}[ht]
    \centering
    \includegraphics{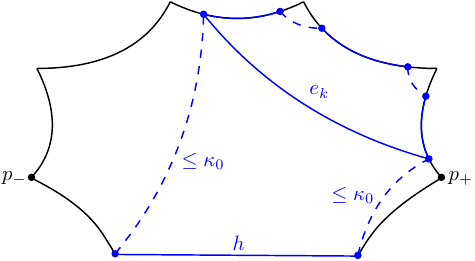}
    \caption{Illustration of Lemma~\ref{lem:h_conn_to_ek}.}
    \label{fig:h0_ek}
\end{figure}

\begin{proof}
    We take the constants
    \begin{itemize}
        \item \(\kappa_1 = \kappa(\lambda,c,0)\) and \(\kappa_2 = \kappa(4,c_0,0)\), obtained by applying Proposition~\ref{prop:bcp} to \((\lambda,c)\)- and \((4,c_0)\)-quasigeodesics with the same endpoints respectively;
        \item \(N_1 = \max\{N_0, 2M + 9L + 2\kappa_1 + 2\kappa_2 + 2\Theta\} + 1\), where \(N_0 = N_0(M)\) is the constant of Lemma~\ref{lem:comp_of_parab_close_to_geodesic} and \(L\) is the constant of Proposition~\ref{prop:osin_polygon};
        \item \(\kappa_0 = \kappa_1 + \rho + 3L\), where \(\rho = \rho(C)\) is the constant of Lemma~\ref{lem:short_ending_components}.
    \end{itemize}
    Suppose that \(\abs{p}_X \geq N_1\).
    First observe that \(N_1\) is greater than \(N_0\), so that by Lemma~\ref{lem:comp_of_parab_close_to_geodesic}, \(h\) is connected to an \(H_\nu\)-component \(h'\) of a geodesic \([p_-,p_+]\) (see Figure~\ref{fig:h0_h'_h''_h'''}) with
    \begin{equation}
    \label{eq:h_h'_close}
        d_X(h_-,h'_-) \leq 3L \quad \textrm{ and } \quad d_X(h_+,h'_+) \leq 3L.
    \end{equation}
    As \(\elem{p} = bhb^{-1}\), the triangle inequality gives us that
    \begin{equation}
    \label{eq:len_of_h0}    
    \begin{split}
        \abs{h}_X &\geq \abs{p}_X - 2\abs{b}_X \\
            & \geq N_1 - 2M \\
            & \geq 9L + 2\kappa_1 + 2\kappa_2 + 2\Theta + 1.
    \end{split}
    \end{equation}
    Combining (\ref{eq:h_h'_close}) and (\ref{eq:len_of_h0}) yields that \(\abs{h'}_X \geq \abs{h}_X - 6L \geq \kappa_1\).
    Moreover, by Proposition~\ref{prop:shortcutting_quasigeodesic}, \(\Sigma(p,\Theta)\)  is \((\lambda,c)\)-quasigeodesic without backtracking.
    Therefore Proposition~\ref{prop:bcp} tells us that there is an \(H_\nu\)-component \(h''\) of \(\Sigma(p,\Theta)\) connected to \(h'\) such that
    \begin{equation}
    \label{eq:h'_h''_close}
        d_X(h'_-,h''_-) \leq \kappa_1 \quad \textrm{ and } \quad d_X(h'_+,h''_+) \leq \kappa_1.
    \end{equation}
    
    Suppose for a contradiction that \(h''\) is an \(H_\nu\)-component of \(f_k\) for some \(k = 0, \dots, m\).
    Let \(p'\) be the subpath of \(p\) with \(p'_- = (f_k)_-\) and \(p'_+ = (f_k)_+\).
    Lemma~\ref{lem:double_shortcutting} tells us that \(\Sigma_0(p',1)\) is \((4,c_0)\)-quasigeodesic without backtracking.
    We have that \(\abs{h''}_X \geq \abs{h}_X - 6L - 2\kappa_1 \geq \kappa_2\) by combining equations (\ref{eq:h_h'_close}), (\ref{eq:len_of_h0}), and (\ref{eq:h'_h''_close}).
    Hence by Proposition~\ref{prop:bcp}, \(h''\) is connected to an \(H_\nu\)-component \(h'''\) of \(\Sigma_0(p',1)\) with
    \begin{equation}
    \label{eq:h''_h'''_close}
        d_X(h''_-,h'''_-) \leq \kappa_2 \quad \textrm{ and } \quad d_X(h''_+,h'''_+) \leq \kappa_2
    \end{equation}
    as in Figure~\ref{fig:h0_h'_h''_h'''}.

    \begin{figure}[ht]
        \centering
        \includegraphics{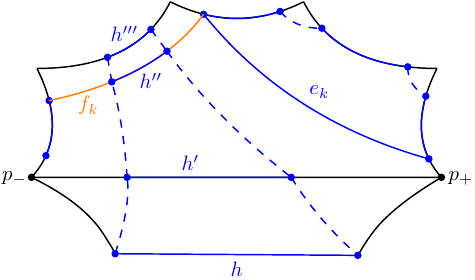}
        \caption{Illustration of proof of Lemma~\ref{lem:h_conn_to_ek}.}
        \label{fig:h0_h'_h''_h'''}
    \end{figure}
    
    By the triangle inequality and equations (\ref{eq:h_h'_close})-(\ref{eq:h''_h'''_close}), we have
    \[
        \abs{h'''}_X \geq \abs{h}_X - 6L - 2\kappa_1 - 2\kappa_2 \geq 3L + 2\Theta_1 + 1,
    \]
    whereas by Lemma~\ref{lem:double_shortcutting}, \(\abs{h'''}_X \leq 3L + 2\Theta\), a contradiction.
    Therefore \(h''\) cannot be an \(H_\nu\)-component of \(f_k\).
    It follows that \(h''\) is a component of \(\Sigma(p,\Theta)\) containing \(e_k\) for some \(k = 1, \dots, m\) and thus that \(h\) is connected to \(e_k\) (as in Figure~\ref{fig:h0_ek}).
    
    It remains to show the inequality in the lemma statement.
    Following Remark~\ref{rem:comp_of_geod_is_an_edge}, \(h''\) consists of at most three edges, one being \(e_k\) and the (possible) other two being edges, respectively the last and the first, of the geodesics \(f_{k-1}\) and \(f_k\).
    Lemma~\ref{lem:short_ending_components} then implies that
    \begin{equation}
    \label{eq:h''_ek_close}
        d_X(h''_-,(e_k)_-) \leq \rho \quad \textrm{ and } \quad d_X(h''_+,(e_k)_+) \leq \rho.
    \end{equation}
    Finally, (\ref{eq:h_h'_close}), (\ref{eq:h'_h''_close}), and (\ref{eq:h''_ek_close}) together with the choice of \(\kappa_0\) give the inequalities 
    \[
        d_X(h_-,(e_k)_-) \leq \kappa_0 \quad \textrm{ and } \quad d_X(h_+,(e_k)_+) \leq \kappa_0.
    \]
    This concludes the lemma.
\end{proof}

\begin{lemma}
\label{lem:parab_short_one_ek}
    For any \(M \geq 0\), there is a constant \(\eta_0 = \eta_0(M) \geq 0\) such that for any \(\Theta \geq \zeta = \zeta(\eta_0,C)\) and \(B \geq E_1(\Theta)\) the following is true.
    
    Let \(b \in G\) with \(\abs{b}_X \leq M\).
    Let \(p = p_1 \dots p_n\) a be a \((B,C,\zeta,\Theta)\)-tamable broken line, and suppose that \(\elem{p} \in b H_\nu b^{-1}\) for some \(\nu \in \Nu\).
    Suppose that \(\abs{b}_{X\cup\mathcal{H}}\) is minimal among elements of \(bH_\nu\) and denote by \(\Sigma(p,\Theta) = f_0 e_1 f_1 \dots f_{m-1} e_m f_m\) the \(\Theta\)-shortcutting of the path \(p\).
    If \(\abs{p}_X \geq N_1\) (where \(N_1 = N_1(\Theta,M)\) is the constant of Lemma~\ref{lem:h_conn_to_ek}), then \(m = 1\).
\end{lemma}

\begin{proof}
    We fix the following constants:
    \begin{itemize}
        \item \(\kappa_1 = \kappa(\lambda,c,0)\) and \(\kappa_2 = \kappa(1,0,3L)\), the constants obtained by applying Proposition~\ref{prop:bcp} to \((\lambda,c)\)-quasigeodesics with the same endpoints and \(3L\)-similar geodesics respectively;
        \item \(\xi = \xi(1,0,M+1)\) is provided by Lemma~\ref{lem:qgds_with_long_comps};
        \item \(\eta_0 = \xi + 2\kappa_1 + 2 \kappa_2 \geq 0\);
    \end{itemize}

    \begin{figure}[hb]
        \centering
        \includegraphics{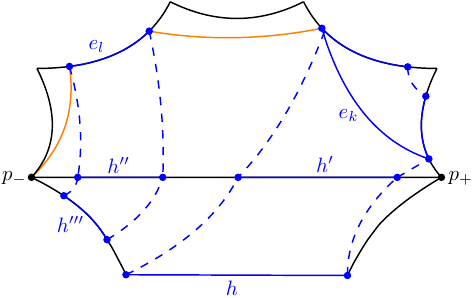}
        \caption{Illustration of  Lemma~\ref{lem:parab_short_one_ek}.}
        \label{fig:one_ek}
    \end{figure}

    Since \(\elem{p} \in b H_\nu b^{-1}\), denote by \(h\) the \(H_\nu\)-edge with \(h_- = p_- b\) and \(\elem{p} = b\elem{h}b^{-1}\).
    Lemma~\ref{lem:h_conn_to_ek} tells us that \(h\) is connected to \(e_k\) for some \(k = 1, \dots, m\), so \(m \geq 1\).
    Moreover, by Lemma~\ref{lem:comp_of_parab_close_to_geodesic}, \(h\) is connected to an \(H_\nu\)-component \(h'\) of a geodesic \([p_-,p_+]\) and
    \[
        d_X(h_-,h'_-) \leq 3L \quad \textrm{ and } \quad d_X(h_+,h'_+) \leq 3L.
    \]
    In particular, this implies that \([p_-,h_-]\) and \([p_-,h'_-]\) are \(3L\)-similar, and as are \([h_+,p_+]\) and \([h'_+,p_+]\).
    
    Suppose for a contradiction that \(m > 1\), so that there is \(l \ne k\) with \(1 \leq l \leq m\).
    By Proposition~\ref{prop:shortcutting_quasigeodesic}, the shortcutting \(\Sigma(p,\Theta)\) is \((\lambda,c)\)-quasigeodesic without backtracking, and further the \(\mathcal{H}\)-component \(e'_l\) of \(\Sigma(p,\Theta)\) containing \(e_l\) satisfies the inequality 
    \begin{equation}
    \label{eq:len_of_e'l}
        \abs{e'_l}_X \geq \eta_0.
    \end{equation}
    Now by Proposition~\ref{prop:bcp}, \(e'_l\) is connected to an \(\mathcal{H}\)-component \(h''\) of the geodesic \([p_-,p_+]\) with
    \begin{equation}
    \label{eq:h''_e'l_k1-sim}
        d_X(h''_-,(e'_l)_-) \leq \kappa_1 \quad \textrm{ and } \quad d_X(h''_+,(e'_l)_+) \leq \kappa_1.
    \end{equation}
    Since \(\Sigma(p,\Theta)\) is without backtracking, \(h''\) must be distinct from \(h'\): if not, then \(e'_l\) and \(e'_k\) would be connected \(\mathcal{H}\)-components of \(\Sigma(p,\Theta)\).
    
    We consider only the case that \(h''\) is an \(\mathcal{H}\)-component of the subpath \([p_-,h'_-]\) of \([p_-,p_+]\), with the other case being dealt with identically.
    It follows from (\ref{eq:len_of_e'l}), (\ref{eq:h''_e'l_k1-sim}), and the definition of \(\eta_0\) that \(\abs{h''}_X \geq \kappa_2\).
    Since \([p_-,h'_-]\) and \([p_-,h_-]\) are \(3L\)-similar geodesics, Proposition~\ref{prop:bcp} tells us that \(h''\) is connected to an \(\mathcal{H}\)-component \(h'''\) of \([p_-,h_-]\) (respectively \([h_+,p_+]\)) and \(h''\) and \(h'''\) satisfy
    \begin{equation}
    \label{eq:h''_h'''_k2-sim}
        d_X(h''_-,h'''_-) \leq \kappa_2 \quad \textrm{ and } \quad d_X(h''_+,h'''_+) \leq \kappa_2.
    \end{equation}
    Combining (\ref{eq:h''_e'l_k1-sim}), (\ref{eq:h''_h'''_k2-sim}), and (\ref{eq:len_of_e'l}), we see that
    \begin{align*}
        \abs{h'''}_X &\geq \abs{e'_l}_X - d_X(h'''_-,(e'_l)_-) - d_X(h'''_+,(e'_l)_+) \\
            &\geq \eta_0 - 2(\kappa_1 + \kappa_2) \geq \xi,
    \end{align*}
    where the last inequality comes from the definition of \(\eta_0\).
    Now we may apply Lemma~\ref{lem:qgds_with_long_comps} to see that 
    \[
        \abs{b}_X = \abs{[p_-,h_-]}_X \geq M + 1 > M
    \]
    contradicting the fact that \(\abs{b}_X \leq M\).
\end{proof}

\begin{lemma}
\label{lem:parab_short_no_sides}
    For any \(M \geq 0\) and \(\Theta \geq \zeta = \zeta(\eta_0,C)\) (where \(\eta_0 = \eta_0(M)\) is the constant of Lemma~\ref{lem:parab_short_one_ek}) there is \(E_2 = E_2(\Theta,M) \geq 0\) such that for any \(B \geq E_2\) the following is true.
    
    Let \(b \in G\) with \(\abs{b}_X \leq M\).
    Let \(p = p_1 \dots p_n\) be a \((B,C,\zeta,\Theta)\)-tamable broken line, and suppose that \(\elem{p} \in b H_\nu b^{-1}\) for some \(\nu \in \Nu\).
    Suppose that \(\abs{b}_{X\cup\mathcal{H}}\) is minimal among elements of \(bH_\nu\) and denote by \(\Sigma(p,\Theta) = f_0 e_1 f_1 \dots f_{m-1} e_m f_m\) the \(\Theta\)-shortcutting of the path \(p\).
    
    If \(\abs{p}_X \geq N_1\) (where \(N_1 = N_1(\Theta,M)\) is the constant of Lemma~\ref{lem:h_conn_to_ek}) then \((e_1)_-\) is a non-terminal vertex of \(p_1p_2\), and \((e_m)_+\) is a non-initial vertex of \(p_{n-1}p_n\).
    Moreover, if \(\abs{p_1}_X \geq B\) (respectively, \(\abs{p_n}_X \geq B\)), then \((e_1)_-\) is a non-terminal vertex of \(p_1\) (respectively, \((e_m)_+\) is a non-initial vertex of \(p_n\)).
\end{lemma}

\begin{proof}
    Define \(E_2 = \max\{E_1(\Theta), (4M + 8 + c_0)\Theta\}\), where \(E_1(\Theta)\) is the constant obtained from Lemma~\ref{lem:double_shortcutting}, and let \(B \geq E_2\).

    Denote by \(h\) the \(H_\nu\)-edge with \(h_- = p_- b\) and \(\elem{p} = b\elem{h}b^{-1}\).
    By Lemma~\ref{lem:parab_short_one_ek} we have \(m = 1\), and so by Lemma~\ref{lem:h_conn_to_ek}, \(h\) is connected to \(e_1\).

    We prove only the statement involving \((e_1)_-\), for a symmetrical argument shows the corresponding statement for \((e_1)_+\).
    Suppose to the contrary, so that \((e_1)_-\) is a vertex of \(p_i\) for \(i > 2\).
    The subpath \(p'\) of \(p\) with endpoints \(p'_- = p_-\) and \(p'_+ = (e_1)_-\) is a \((4,c_0)\)-quasigeodesic broken line in \(\relcay\) by Lemma~\ref{lem:fk_quasigeodesic}, each \(\mathcal{H}\)-component \(h\) of the segments of which satisfy \(\abs{h}_X \leq \Theta\) by Remark~\ref{rem:cpts_of_p'k}.
    Moreover, \(p_2\) is a subpath of \(p'\) and \(\abs{p_2}_X \geq B \geq E_2\) by tamability condition \ref{cond:tam_1}.
    Then by Lemma~\ref{lem:rel_paths_with_short_comps}, 
    \[
        \ell(p') \geq \ell(p_2) \geq \frac{E_2}{\Theta} \geq 4M + 8 + c_0,
    \] 
    whence by quasigeodesicity of \(p'\) we have
    \begin{equation}
    \label{eq:p'_len_lower_bd}
        \abs{p'}_{X\cup\mathcal{H}} \geq \frac{1}{4}\ell(p') - \frac{c_0}{4} \geq M + 2.
    \end{equation}
    On the other hand, we have \(d_{X\cup\mathcal{H}}(p_-,h_-) \leq \abs{b}_{X\cup\mathcal{H}} \leq \abs{b}_X \leq M\) and that \(d_{X\cup\mathcal{H}}(h_-,(e_1)_-) \leq 1\) since \(h\) and \(e_1\) are connected.
    It follows, then, that:
    \[
        \abs{p'}_{X\cup\mathcal{H}} \leq M + 1,
    \]
    contradicting (\ref{eq:p'_len_lower_bd}).
    This means that \(p'\) cannot contain the entire subpath \(p_2\).
    Hence \(p'_- = (e_1)_-\) must be a non-terminal vertex of \(p_1p_2\).
    If, in addition, \(\abs{p_1}_X \geq B\) the same argument shows that \((e_1)_-\) is a non-terminal vertex of \(p_1\), as \(p_1\) is also a subpath of \(p'\).
\end{proof}


\section{Path representatives}
\label{sec:path_reps}

Let \(Q' \leqslant Q, R' \leqslant R\) and suppose that \(U, V \subseteq Q \cup R\) are nonempty subsets.
Elements of \(U \langle Q', R' \rangle V\) are labels of certain broken lines in \(\relcay\) which can be assigned a numerical invariant.
When this numerical invariant is minimal and \(Q'\) and \(R'\) satisfy certain conditions, these broken lines are tamable (with appropriate parameters, in the sense of Definition~\ref{def:tamable}).

In this section we define path representatives of elements of \(U \langle Q', R' \rangle V\), following \cite{MinMin}.
We then collect a variety of results about such path representatives, and recall the metric conditions that the subgroups \(Q'\) and \(R'\) must satisfy.
Proofs are mostly omitted in this section, as they are virtually identical to the associated ones in \cite{MinMin}.

\begin{definition}[Path representative]
\label{def:path_reps}
    Consider an element \(g \in U \langle Q', R' \rangle V\).
    Let \(p= q_1 p_1 \dots p_n q_2\) be a broken line in \(\Gamma(G,X\cup\mathcal{H})\) with geodesic segments \(q_1, p_1, \dots, p_n\), and \(q_2\) such that
    \begin{itemize}
        \item \(\elem{p} = g\)
        \item \(\elem{p_i} \in Q' \cup R'\) for each \(i \in \{1,\dots,n\}\)
        \item \(\elem{q_1} \in U\) and \(\elem{q_2} \in V\)
    \end{itemize}
    We will call \(p\) a \emph{path representative} of \(g\).
\end{definition}

Observe that when \(U = V = \{1\}\), we essentially recover the definition of \cite{MinMin} of path representatives for elements of \(\langle Q', R' \rangle\).
Indeed, in this case the initial and final segments \(q_1\) and \(q_2\) in the above definition are forced to be trivial (i.e. paths with only a single vertex), and we may omit their mention.
Moreover, with \(U = Q\) and \(V = R\) we obtain the path representatives of elements of \(Q \langle Q', R' \rangle R\) referenced in \cite[Definition 10.6]{MinMin}.

To choose an optimal path representative we define their types.

\begin{definition}[Type of a path representative] 
\label{def:type_of_path_rep}
    Suppose that $p= q_1 p_1\dots p_n q_2$ is a broken line in $\relcay$.
    Let \(T\) denote the set of all \(\mathcal{H}\)-components of the segments of \(p\). 
    We define the \emph{type} $\tau(p)$ of $p$ to be the triple 
    \[
        \tau(p)=\Big(n,\ell(p),\sum_{t \in T} |t|_X \Big) \in {\NN_0}^3.
    \]
\end{definition}

\begin{definition}[Minimal type]
    Given \(g \in U \langle Q', R' \rangle V\), the set $\mathcal S$ of all path representatives of $g$ is non-empty. 
    Therefore the subset $\tau(\mathcal{S})=\{\tau(p) \mid p \in \mathcal{S}\} \subseteq {\NN_0}^3$, where ${\NN_0}^3$ is equipped with the lexicographic order, will have a unique minimal element.

    We will say that $p= q_1 p_1 \dots p_n q_2$ is a \emph{path representative of $g$ of minimal type} if $\tau(p)$ is the minimal element of $\tau(\mathcal S)$.
\end{definition}

\begin{remark}
\label{rem:alt}
    Note that if \(p_1\) and \(p_2\) are paths with \((p_1)_+ = (p_2)_-\) whose labels both represent elements of \(Q'\) (or, respectively, both \(R'\)), then the label of any geodesic \([(p_1)_-,(p_2)_+]\) also represents an element of \(Q'\) (respectively, \(R'\)).
    Hence in a path representative of $g \in U \langle Q',R' \rangle V$ of minimal type, the labels of the consecutive segments necessarily alternate between representing elements of \(Q' \setminus S\) and \(R' \setminus S\), whenever \(g\) is not itself an element of \(UQ'V\) or \(UR'V\).
\end{remark}

\subsection{Metric conditions and path representatives of minimal type}

\begin{notation}
    Let \(U \subseteq G\) be a subset of \(G\). We write
    \[
        \minx U = \min \{ \abs{u}_X \, | \,  u \in U\}
    \]
    when \(U\) is nonempty, and \(\minx U = \infty\) otherwise.
\end{notation}

Recall the following metric conditions pertaining to subgroups \(Q' \leqslant Q\) and \(R' \leqslant R\) and family of maximal parabolic subgroups \(\mathcal{P}\) in \(G\).

\begin{itemize}
    \descitem{C1} \(Q' \cap R' = S\);
    \descitem{C2} \(\minx(Q \langle Q', R'\rangle Q \setminus Q ) \geq B\) and \(\minx(R \langle Q', R'\rangle R \setminus R ) \geq B\);
    \descitem{C3} \(\minx \Bigl( (PQ' \cup PR') \setminus PS\Bigr) \geq C\), for each $P \in \mathcal{P}$.
    \descitem{C4} \(Q_P \cap \langle Q_P', R_P' \rangle = Q'_P\) and \(R_P \cap \langle Q_P', R_P' \rangle = R_P'\), for every $P \in \mathcal{P}$;
    \descitem{C5}  \(\minx \Bigl(q \langle Q_P', R_P' \rangle R_P \setminus qQ'_P R_P\Bigr) \geq C \), for each $P \in \mathcal{P}$ and all $q \in Q_P$,
\end{itemize}
where for a subgroup \(H \leqslant G\) and \(P \in \mathcal{P}\), \(H_P\) denotes the subgroup \(H \cap P\).

The following is a straightforward observation about condition \descref{C2} that we will find useful.

\begin{lemma}[{\cite[Lemma 10.1]{MinMin}}]
\label{lem:C2->C2_old}
    Suppose that \(Q' \leqslant Q\) and \(R' \leqslant R\) satisfy \descref{C2} with constant \(B \geq 0\). Then
    \[
        \minx \Bigl( (Q' \cup R') \setminus S \Bigr) \geq B.
    \]
\end{lemma}

We can leverage separability properties of \(G\) to find finite index subgroups of \(Q\) and \(R\) satisfying the above conditions for any given constants \(B\) and \(C\) and finite family \(\mathcal{P}\).

\begin{proposition}[{\cite[Proposition 14.3]{MinMin}}]
\label{prop:sep->C1-C5}
    Let \(G\) be a finitely generated QCERF relatively hyperbolic group with double coset separable peripheral subgroups, and let \(Q\) and \(R\) be finitely generated relatively quasiconvex subgroups.
    
    For any \(B \geq 0, C \geq 0\), and finite set \(\mathcal{P}\) of maximal parabolic subgroups of \(G\), there is a family of pairs of finite index subgroups \(Q' \leqslant_f Q\) and \(R' \leqslant_f R\) as in \descref{E} satisfying \descref{C1}-\descref{C5} with constants \(B\) and \(C\) and set \(\mathcal{P}\).
\end{proposition}

Below we collect some results demonstrating useful properties of minimal type path representatives of elements of \(U \langle Q', R' \rangle V\).
We emphasise that the proofs of the following statements are very similar to the associated statements in \cite{MinMin}.

\begin{lemma}[{\cite[Lemma 6.7]{MinMin}}]
\label{lem:bdd_inn_prod}
    There is a constant \(C_0 \geq 0\) such that the following holds.

    Let $Q' \leqslant Q$ and $R' \leqslant R$ be subgroups satisfying condition \descref{C1}, and let \(q \in Q\) and \(r \in R\).
    Suppose that either \(U = qQ'\) or \(U = rR'\) and let \(V = U^{-1}\). 
    If \(p=q_1 p_1 \dots p_n q_2\) is a minimal type path representative of an element \(g \in U \langle Q', R' \rangle V\) and $f_0, \dots, f_{n+2} \in G$ are the nodes of $p$ then \(\langle f_{i-1}, f_{i+1} \rangle_{f_i}^{rel} \leq C_0\) for each \(i= 1, \dots, n+1\).
\end{lemma}

\begin{lemma}[{\cite[Lemma 7.3]{MinMin}}]
\label{lem:short_cusps}
    There exists a constant \(C_1 \ge 0\) satisfying the following.

    Let \(Q' \leqslant Q\) and \(R' \leqslant R\) be subgroups satisfying condition \descref{C1}, and let \(q \in Q\) and \(r \in R\).
    Suppose that either \(U = qQ'\) or \(U = rR'\), let \(V = U^{-1}\) and suppose that \(p\) is a minimal type path representative for an element \(g \in U \langle Q',R' \rangle V\). 
    If \(s\) and \(t\) are connected \(\mathcal{H}\)-components of adjacent segments \(a\) and \(b\) of \(p\) respectively, then \(d_X(s_+, a_+) \leq C_1\) and \(d_X(a_+,t_-) \leq C_1\).
\end{lemma}

\begin{notation}
\label{not:Pm}
    Let \(M \geq 0\) and let \(C_1 \geq 0\) be the constant of Lemma~\ref{lem:short_cusps}. We define the following finite collection of maximal parabolic subgroups of \(G\):
    \[
        \mathcal{P}_M = \{b H_\nu b^{-1} \mid b \in G, \nu \in \Nu, \abs{b}_X \leq M + C_1\}.
    \]  
\end{notation}

If \(Q'\) and \(R'\) satisfy conditions \descref{C1}-\descref{C5} with sufficiently large constants \(B\) and \(C\), then minimal type path representatives of elements of \(U \langle Q', R' \rangle V\) are tamable.
We define the constant \(C'_0 = \max \{C_0, 14 \delta\}\), where \(C_0\) is the constant of Lemma~\ref{lem:bdd_inn_prod}.

\begin{lemma}[{\cite[Lemma 10.3]{MinMin}}]
\label{lem:pathreps_are_tamable}
    There are constants \(\lambda = \lambda(C'_0), c = c(C'_0)\) and, for each \(\eta \geq 0\), constants \(C_2 = C_2(\eta) \geq 0\), $\zeta=\zeta(\eta) \ge 1$, \(\Theta_1 = \Theta_1(\eta) \in \NN\) with \(\Theta_1 \geq \zeta\)  and  \(B_1 = B_1(\eta) \geq 0\)  such that the following is true.

    Suppose that \(Q'\leqslant Q\) and \(R'\leqslant R\) are subgroups satisfying conditions \descref{C1}-\descref{C5} with constants \(B \geq B_1\) and \(C \geq C_2\) and family \(\mathcal{P} \supseteq \mathcal{P}_0\). 
    Let \(q \in Q, r \in R\), suppose that either \(U = qQ'\) or \(U = rR'\), and let \(V = U^{-1}\).
    If \(p = q_1 p_1 \dots p_n q_2\) is a minimal type path representative for an element \(g \in U \langle Q', R' \rangle V\) then \(p\) is \((B,C'_0,\zeta,\Theta_1)\)-tamable.

    Moreover, let \(\Sigma(p,\Theta_1) = f_0 e_1 f_1 \dots f_{m-1} e_m f_m\) be the \(\Theta_1\)-shortcutting of \(p\) obtained from Procedure~\ref{proc:shortcutting}, and let $e_k'$ be the \(\mathcal{H}\)-component of \(\Sigma(p,\Theta_1)\) containing \(e_k\), $k=1,\dots,m$. Then \(\Sigma(p,\Theta_1)\) is a \((\lambda,c)\)-quasigeodesic without backtracking and  \(\abs{e'_k}_X \geq \eta\), for each \(k = 1, \dots, m\).
\end{lemma}

We note that the statement of the above lemma in \cite{MinMin} does not include the fact that \(\Theta_1 \geq \eta\) as stated here, though \(\Theta_1\) is explicitly constructed to be at least \(\zeta\) in its proof.

The following gives us a way of constructing parabolic paths that in some way approximate an instance of consecutive backtracking in a path representative.
The statement is a modification of \cite[Lemma 8.3]{MinMin} that allows for an extra parameter.
For completeness we include a proof.

\begin{lemma}
\label{lem:(c3)->vertex_constr} 
    Let \(M \geq 0\) and suppose that subgroups $Q' \leqslant Q$ and $R' \leqslant R$ satisfy conditions \descref{C1} and \descref{C3} with constant \(C\) and family \(\mathcal{P}  \) such that $C \ge M + C_1 + 1$ and \(\mathcal{P} \supseteq \mathcal{P}_M\) as in Notation~\ref{not:Pm}.
    Let $P=bH_\nu b^{-1} \in\mathcal{P}_M$, for some $\nu \in \Nu$ and $b \in G$, with $|b|_X \le M$, and let $p$ be a path in $\relcay$ with $\elem{p} \in Q' \cup R'$.

    Suppose that there is a vertex $v$ of $p$ and an element $u \in P$ satisfying $u^{-1}p_- \in S$, $v \in Pb$, and $d_X(v,p_+) \le C_1$. 
    Then there exists a geodesic path $p'$ such that $(p')_-=u$, $\elem{p'} \in P$, and $(p')_+^{-1}p_+ \in S$. 
    In particular, if $\elem{p} \in Q'$ (respectively, $\elem{p} \in R'$) then $\elem{p'} \in Q' \cap P$ (respectively, $\elem{p'} \in R' \cap P$).
\end{lemma}

\begin{proof} 
    Denote $x=p_-$,  $y=p_+$ and $z=vb^{-1} \in P$.
    Then $u^{-1}z \in P$ and $x^{-1}y=\elem{p} \in Q' \cup R'$.

    Since $u^{-1} x \in S = Q' \cap R'$, we obtain 
    \[
        u^{-1}y=(u^{-1}x) (x^{-1}y) \in Q' \cup R',
    \]
    whence $z^{-1}y=(z^{-1}u )(u^{-1}y) \in P (Q' \cup R')$. 
    Now, observe that
    \[
        |z^{-1} y|_X=d_X(z,y) \le d_X(z,v)+d_X(v,y) \le |b|_X+C_1 \le M + C_1 < C.
    \] 
    Condition \descref{C3} now implies that $z^{-1}y \in PS $, i.e., $z^{-1}y =fh$, for some $f \in P$ and $h \in S$.
    Let $p'$ be a geodesic path starting at $u$ and ending at $zf \in P$. Then $\elem{p'}=u^{-1}zf \in P$ and
    \[
        (p')_+^{-1}p_+=f^{-1}z^{-1}y=h \in S.
    \]
    The last statement of the lemma follows from \descref{C1} and the observation that 
    \[
        \elem{p'}=u^{-1}(p')_+=u^{-1} p_- \, \elem{p} \, (p_+)^{-1}(p')_+ \in S\,\elem{p}\, S. \qedhere
    \]
\end{proof}


\section{Reduction to the short conjugator case}
\label{sec:short_conj}

In this section we again follow the notation of Convention~\ref{conv:GQR}.
We will first prove the special case of the main result for conjugates of the peripheral subgroups by uniformly short elements.
In this case, taking \(Q'\) and \(R'\) with sufficiently deep index, the conjugator \(u \in \langle Q', R' \rangle\) in the statement of Theorem~\ref{thm:parab_join} will be trivial.
In particular, we will prove the following:

\begin{proposition}
\label{thm:parab_join_short_conj}
    For any \(M \geq 0\) there exist constants \(B_2 = B_2(M) \geq 0\) and \(C_3 = C_3(M) \geq 0\) such that the following is true.
    
    Suppose \(Q' \leqslant Q\) and \(R' \leqslant R\) satisfy conditions \descref{C1}-\descref{C5} with constants \(B \geq B_2, C \geq C_3\), and family \(\mathcal{P}_M\) (see Notation~\ref{not:Pm}).
    If \(P \in \mathcal{P}_M\) is such that \(\langle Q', R' \rangle \cap P\) is infinite, then
    \[
        \langle Q', R' \rangle \cap P = \langle Q' \cap P, R' \cap P \rangle.
    \]
\end{proposition}

\begin{proof}
    Let \(P \in \mathcal{P}_M\) and suppose that \(\langle Q', R' \rangle \cap P\) is infinite.
    We will fix the following notation for the proof:
    \begin{itemize}
        \item \(P = b H_\nu b^{-1}\) where \(\nu \in \Nu\) and \(b \in G\) with \(\abs{b}_X \leq M\);
        \item \(b_1 \in b H_\nu\) has minimal length with respect to \(d_{X\cup\mathcal{H}}\) and \(\abs{b_1}_X \leq \xi_0 M^2\) (as in Remark~\ref{rem:no_ending_cpts}), where \(\xi_0\) is the constant of Lemma~\ref{lem:len_of_subgeodesic};
        \item \(C'_0 = \max\{C_0, 14 \delta\}\), where \(C_0\) is the constant of Lemma~\ref{lem:bdd_inn_prod};
        \item \(\eta_0 = \eta_0(\xi_0 M^2)\) is the constant of Lemma~\ref{lem:parab_short_one_ek};
        \item \(B_1 = B_1(\eta_0), C_2 = C_2(\eta_0)\), and \(\Theta_1 = \Theta_1(\eta_0)\) are the constants obtained from Lemma~\ref{lem:pathreps_are_tamable} applied with \(\eta_0\);
        \item \(N_1 = N_1(\Theta_1,\xi_0 M^2)\) and \(\kappa_0 = \kappa_0(C'_0)\) are the constants of Lemma~\ref{lem:h_conn_to_ek};
        \item \(B_2 = \max\{B_1, E_2(\Theta_1,\xi_0 M^2)\}\), where \(E_2(\Theta_1,\xi_0 M^2)\) is the constant of Lemma~\ref{lem:parab_short_no_sides} and \(C_3 = \max\{C_2, M + C_1 + 1\}\), where \(C_1\) is the constant of Lemma~\ref{lem:short_cusps}.
    \end{itemize}
    By assumption, \(\langle Q', R' \rangle \cap P\) is infinite, so there is an element \(g \in \langle Q', R' \rangle \cap P\) with \(\abs{g}_X \geq N_1\).
    Let \(p = p_1 \dots p_n\) be a path representative of minimal type for \(g\) (as an element of \(U \langle Q', R' \rangle V\), with \(U = V = \{1\}\)) with \(p_- = 1\).
    If \(n = 1\) then \(g = \elem{p} \in (Q' \cup R') \cap P\) and we are done, so suppose that \(n > 1\).
    We write \(h\) for the \(H_\nu\)-edge of \(\relcay\) with \(h_- = b_1\) and \(g = b_1\elem{h}b_1^{-1}\).
    
    We consider the shortcutting \(\Sigma(p,\Theta_1) = f_0 e_1 f_1 \dots f_{m-1} e_m f_m\) of \(p\) obtained from Procedure~\ref{proc:shortcutting}.
    Lemma~\ref{lem:C2->C2_old}, together with the fact that \(p\) is minimal and \(n > 1\), gives us that \(\abs{p_i}_X \geq B\) for each \(i = 1, \dots, n\).
    Moreover, Lemma~\ref{lem:pathreps_are_tamable} gives that \(p\) is \((B,C'_0,\zeta,\Theta_1)\)-tamable.
    Lemmas~\ref{lem:parab_short_one_ek} and \ref{lem:parab_short_no_sides} tell us that \(m = 1\) and that \((e_1)_-\) and \((e_1)_+\) are non-terminal and non-initial vertices of \(p_1\) and \(p_n\) respectively.
    As such, we may suppose that \(f_0\) and \(f_1\) are chosen to be subpaths of the geodesics \(p_1\) and \(p_n\) respectively, so that \(e_1\) is an \(\mathcal{H}\)-component of \(\Sigma(p,\Theta_1)\).
    Moreover, Lemma~\ref{lem:h_conn_to_ek} implies that \(e_1\) is connected to \(h\) with
    \begin{equation}
         d_X(h_-,(e_1)_-) \leq \kappa_0 \quad \textrm{ and } \quad d_X(h_+,(e_1)_+) \leq \kappa_0.
    \end{equation}
    It follows that \((e_1)_- H_\nu = b_1 H_\nu = b H_\nu\).
    Denote by \(h_1, \dots, h_n\) the pairwise connected \(H_\nu\)-components of the segments \(p_1, \dots, p_n\) that constitute the instance of consecutive backtracking associated to \(e_1\).
    
    We will inductively construct a sequence of paths \(p'_1, \dots, p'_{n-1}\) (cf. \cite[Proposition 8.4]{MinMin})  with the following properties:
    \begin{itemize}
        \item \((p'_1)_- = 1\);
        \item \(\elem{p'_i} \in (Q' \cup R') \cap P \) for each \(i = 1, \dots, n-1\);
        \item \((p'_i)^{-1}_+ (p_i)_+ \in S\) for each \(i = 1, \dots, n-1\).
    \end{itemize}
    
    It is straightforward to verify using Lemma~\ref{lem:short_cusps} that \(p_1\) satisfies the hypotheses of Lemma~\ref{lem:(c3)->vertex_constr} with \(u = 1, v = (h_1)_+\), and subgroup \(b H_\nu b^{-1}\).
    Thus there is \(p'_1\) with \((p'_1)_- = 1, \elem{p'_1} \in (Q' \cup R') \cap P\), and \((p'_1)_+^{-1} (p_1)_+ \in S\).
    Similarly, for any \(1 < i \leq n-1\), we can verify that we can apply Lemma~\ref{lem:(c3)->vertex_constr} to the path \(p_i\) with \(u = (p'_{i-1})_+, v = (h_i)_+,\) and \(P = b H_\nu b^{-1}\).
    We thus obtain a path \(p'_i\) with \((p'_i)_- = (p'_{i-1})_+\), \(\elem{p'_i} \in (Q' \cup R') \cap P\), and \((p'_i)^{-1} p_i \in S\).
    
    We will write \(z = (p'_{n-1})_+ = \elem{p'_1} \dots \elem{p'_{n-1}} \in \langle Q' \cap P, R' \cap P \rangle\).
    Since \(g \in P\) and \(z \in P\), it is also true that \(z^{-1} g \in P\). 
    Moreover, 
    \begin{align*}
        z^{-1} g &= z^{-1} (p_{n-1})_+ (p_{n-1})_+^{-1} g \\
            &=  ((p'_{n-1})_+^{-1} (p_{n-1})_+) \elem{p_n} \in S(Q' \cup R') = Q' \cup R',
    \end{align*}
    so that \(z^{-1} g \in (Q' \cap P) \cup (R' \cap P)\).
    Thus \(g = zz^{-1}g \in \langle Q' \cap P, R' \cap P \rangle\).
    
    Since \(g\) was an arbitrary element of \(\langle Q', R' \rangle \cap P\) with \(\abs{g}_X \geq N_1\), we have shown that all but finitely many elements of \(\langle Q', R' \rangle \cap P\) lie in \(\langle Q' \cap P, R' \cap P \rangle\).
    Now applying Lemma~\ref{lem:almost_containment_inf_subgps} shows that the former subgroup is contained in the latter.
    The reverse inclusion is immediate.
\end{proof}

To complete the proof of Theorem~\ref{thm:parab_join}, we reduce computation of the subgroup \(\langle Q', R' \rangle \cap P\), where \(P = bH_\nu b^{-1}\) is an arbitrary maximal parabolic subgroup, to the case when \(P\) belongs to a fixed finite set of maximal parabolic subgroups.
An application of Proposition~\ref{thm:parab_join_short_conj} will then yield the general result.

To do this, we observe that when \(\langle Q', R' \rangle \cap P\) is infinite, the conjugator \(b\) has a decomposition as an element of \(\langle Q', R' \rangle Q x\) or \(\langle Q', R' \rangle R x\) where \(x \in G\) has uniformly bounded length with respect to \(d_X\).
Thus, up to conjugation by an element in \(\langle Q', R' \rangle\), we need only consider intersections of the form \(\langle Q', R' \rangle \cap sx H_\nu x^{-1}s^{-1}\), where \(s \in Q \cup R\) and \(\nu \in \Nu\).

\begin{lemma}
\label{lem:pass_to_q_conj}
    There are constants \(B_3 \geq 0\) and \(\sigma \geq 1\) such that the following is true.
    
    Suppose \(Q' \leqslant Q\) and \(R' \leqslant R\) satisfy conditions \descref{C1}-\descref{C5} with constants \(B \geq B_3, C \geq C_2(1)\) (where \(C_2(1)\) is obtained from Lemma~\ref{lem:pathreps_are_tamable}) and family \(\mathcal{P} \supseteq \mathcal{P}_0\) (as in Notation~\ref{not:Pm} with \(M = 0\)).
    Let \(P = b H_\nu b^{-1}\) be a maximal parabolic subgroup, with \(\abs{b}_{X\cup\mathcal{H}}\) minimal among elements of \(bH_\nu\).
    
    Suppose that \(\langle Q', R' \rangle \cap P\) is infinite.
    Then there are elements \(s \in Q \cup R, u \in \langle Q', R' \rangle\), and \(x \in G\) such that \(b = usx\) and \(\abs{x}_X \leq \sigma\).
    In particular,
    \[
        \langle Q', R' \rangle \cap P = u \Big(\langle Q', R' \rangle  \cap sxH_\nu x^{-1}s^{-1} \Big) u^{-1},
    \]
    and \(u s x H_\nu x^{-1} s^{-1} u^{-1} = P\).
    Moreover, if \(Q' \cap P\) or \(R' \cap P\) is infinite, then we may take \(u = 1\) in the above.
\end{lemma}

\begin{figure}[ht]
    \centering
    \includegraphics{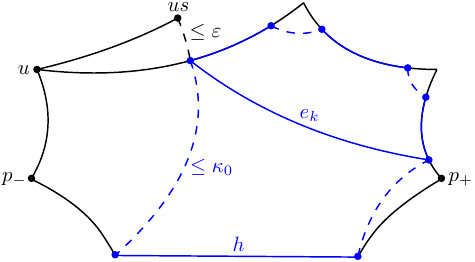}
    \caption{Illustration of Lemma~\ref{lem:pass_to_q_conj}.}
    \label{fig:quad}
\end{figure}

\begin{proof}
    We define the following notation for this proof:
    \begin{itemize}
        \item \(C'_0 = \max\{C_0, 14 \delta\}\), where \(C_0\) is the constant of Lemma~\ref{lem:bdd_inn_prod};
        \item \(\zeta = \zeta(1), \Theta_1 = \Theta_1(1), B_1 = B_1(1),\) and \(C_2 = C_2(1)\) are the constants of Lemma~\ref{lem:pathreps_are_tamable};
        \item \(E_1 = E_1(\Theta_1)\) is the constant of Lemma~\ref{lem:double_shortcutting};
        \item \(N_1 = N_1(\Theta_1,\abs{b}_X)\) and \(\kappa_0 = \kappa_0(C'_0)\) are obtained from Lemma~\ref{lem:h_conn_to_ek};
        \item \(B_3 = \max\{B_1, E_1\}\) and \(\sigma = \max\{\kappa_0 + \varepsilon, 1\}\).
    \end{itemize}
    
    Since \(\langle Q', R' \rangle \cap P\) is infinite, there is an element \(g \in \langle Q', R' \rangle\) with \(\abs{g}_X \geq N_1\).
    Let \(p = p_1 \dots p_n\) be a minimal type path representative of \(g\) (as an element of \(U \langle Q', R' \rangle V\), with \(U = V = \{1\}\)) with \(p_- = 1\), and let \(h\) be the \(H_\nu\)-edge of \(\relcay\) such that \(h_- = b\) and \(g = \elem{p} = b \elem{h} b^{-1}\).
    
    By Proposition~\ref{lem:pathreps_are_tamable}, \(p\) is \((B,C'_0,\zeta,\Theta_1)\)-tamable. 
    Denote by \(\Sigma(p,\Theta_1) = f_0 e_1 f_1 \dots f_{m-1} e_m f_m\) the \(\Theta_1\)-shortcutting of \(p\) obtained from Procedure~\ref{proc:shortcutting}. 
    Then by Lemma~\ref{lem:h_conn_to_ek}, \(h\) is connected to \(e_k\) for some \(k = 1, \dots, m\) with 
    \begin{equation}
    \label{eq:hek_bd}
        d_X(b,(e_k)_-) = d_X(h_-,(e_k)_-) \leq \kappa_0.
    \end{equation}
    Let \(h_i\) be the \(H_\nu\)-component of a segment \(p_i\) of \(p\), with \((h_i)_- = (e_k)_-\).
    
    Take \(u = (p_i)_- \in \langle Q', R' \rangle\).
    If \(\elem{p_i} \in Q'\) then by quasiconvexity of \(Q\), there is an element \(s \in Q\) such that 
    \begin{equation}
    \label{eq:us-1hi_bd}
        d_X(us,(h_i)_-) \leq \varepsilon.    
    \end{equation}
    Otherwise \(\elem{p_i} \in R'\), whence by the quasiconvexity of \(R\), there is an element \(s \in R\) satisfying the same inequality.
    In either case, take \(x = s^{-1} u^{-1} b\) and observe that combining (\ref{eq:hek_bd}) with (\ref{eq:us-1hi_bd}) gives 
    \[
        \abs{x}_X = d_X(b,us) \leq  d_X(b,(h_i)_-) + d_X(us,(h_i)_-)\leq \kappa_0 + \varepsilon \leq \sigma.
    \]
    
    It is immediate from the definition of \(x\) that \(b = usx\), whence \(u s x H_\nu x^{-1} s^{-1} u^{-1} = b H_\nu b^{-1} = P\).
    It follows that
    \begin{align*}
        u \Big(\langle Q', R' \rangle \cap sxH_\nu x^{-1}s^{-1}\Big) u^{-1} &= u\langle Q', R' \rangle u^{-1} \cap usxH_\nu x^{-1}s^{-1}u^{-1} \\
            &= \langle Q', R' \rangle \cap P,
    \end{align*}
    as required.

    Finally, note that when \(Q' \cap P\) is infinite we may take \(g \in Q' \cap P\) with \(\abs{g}_X \geq N_1\), in which case \(p\) consists of a single geodesic segment.
    Following the above argument in this case gives that \(h_i\) is an \(H_\nu\)-component of this segment and \(u = p_- = 1\).
    The case with \(R' \cap P\) infinite is identical.
\end{proof}

When \(s\) is not an element of \(Q'\) or \(R'\), the element \(sx\) obtained above cannot be further decomposed in a useful way, but it does fit into a sort of dichotomy. 
We find that either \(\langle Q', R' \rangle\) intersects \(sxH_\nu x^{-1} s^{-1}\) in an elementary way, or that \(sx\) is an element of \(Q'yH_\nu\) or \(R'yH_\nu\), where \(y\) has uniformly bounded length with respect to \(d_X\).
This completes the reduction (up to \(\langle Q', R' \rangle\)-conjugacy) of computing \(\langle Q', R' \rangle \cap P\) from arbitrary maximal parabolic \(P \leqslant G\) to finitely many conjugates of \(H_\nu\), for \(\nu \in \Nu\).

\begin{proposition}
\label{prop:back_to_Q'R'}
    There are constants \(B_4, C_4 \geq 0\) and \(\tau \geq \sigma\) (where \(\sigma\) is the constant of Lemma~\ref{lem:pass_to_q_conj}) such that if \(Q' \leqslant Q\) and \(R' \leqslant R\) satisfy \descref{C1}-\descref{C5} with constants \(B \geq B_4, C \geq C_4\) and family \(\mathcal{P} \supseteq \mathcal{P}_0\) (as in Notation~\ref{not:Pm}) then the following is true.
    
    Let \(s \in Q \cup R\), \(x \in G\) with \(\abs{x}_X \leq \sigma\), and \(\nu \in \Nu\). If \(\langle Q', R' \rangle \cap sx H_\nu x^{-1}s^{-1}\) is infinite then one of the following holds:
    \begin{itemize}
        \item \(s \in Q' \cup R'\) and \(\langle Q', R' \rangle \cap sx H_\nu x^{-1}s^{-1} = s \langle Q' \cap x H_\nu x^{-1}, R'\cap x H_\nu x^{-1} \rangle s^{-1}\), or
        \item \(s \in Q\) and \(\langle Q', R' \rangle \cap sx H_\nu x^{-1}s^{-1} = Q' \cap sx H_\nu x^{-1}s^{-1}\), or
        \item \(s \in R\) and \(\langle Q', R' \rangle \cap sx H_\nu x^{-1}s^{-1} = R' \cap sx H_\nu x^{-1}s^{-1}\), or
        \item there are elements \(t \in Q' \cup R'\) and \(y \in G\), with \(\abs{y}_X \leq \tau\), such that \(sx \in tyH_\nu\). In particular,
            \[
                \langle Q', R' \rangle \cap sx H_\nu x^{-1}s^{-1} = t \Big(\langle Q', R' \rangle \cap y H_\nu y^{-1}\Big) t^{-1}
            \]
            and \(t y H_\nu y^{-1} t^{-1} = sx H_\nu x^{-1}s^{-1}\).
    \end{itemize}
\end{proposition}

\begin{proof}
    In this proof we use the following notation:
    \begin{itemize}
        \item \(C_0\) and \(C_1\) are the constants of Lemmas~\ref{lem:bdd_inn_prod} and \ref{lem:short_cusps} respectively, and \(C'_0 = \max\{C_0,14\delta\}\);
        \item \(x_1 \in x H_\nu\) has minimal length with respect to \(d_{X\cup\mathcal{H}}\) and \(\abs{x_1}_X \leq \xi_0 \sigma^2\) (as in Remark~\ref{rem:no_ending_cpts}), where \(\xi_0 \geq 1\) is the constant of Lemma~\ref{lem:len_of_subgeodesic};
        \item \(\eta_0 = \eta_0(\xi_0\sigma^2)\) is the constant of Lemma~\ref{lem:parab_short_one_ek};
        \item \(\zeta = \zeta(\eta_0), \Theta_1 = \Theta_1(\eta_0)\), and \(B_1 = B_1(\eta_0)\) are the constants obtained from Lemma~\ref{lem:pathreps_are_tamable};
        \item \(B_4 = \max\{E_1(\Theta_1), E_2(\Theta_1, \xi_0\sigma^2), B_1, B_2(\sigma)\}\) and \(C_4 = \max\{C_2(\eta_0), C_3(\sigma)\}\), where \(E_1(\Theta_1)\) is the constant of Lemma~\ref{lem:double_shortcutting}, \(E_2(\Theta_1, \xi_0\sigma^2)\) is the constant of Lemma~\ref{lem:parab_short_no_sides}, \(C_2(\eta_0)\) is the constant of Lemma~\ref{lem:pathreps_are_tamable}, and \(B_2(\sigma)\) and \(C_3(\sigma)\) are the constants of Proposition~\ref{thm:parab_join_short_conj};
        \item \(\kappa_0 = \kappa_0(C'_0)\) and \(N_1 = N_1(\Theta_1,\xi_0\sigma^2)\) are the constants of Lemma~\ref{lem:h_conn_to_ek};
        \item \(\tau = \max\{C_1, B_4 + \xi_0 \sigma^2 + \kappa_0\}\).
    \end{itemize}

    If \(s \in Q' \cup R'\), then \(s^{-1} \langle Q', R' \rangle s = \langle Q', R' \rangle\) so that
    \[
        \langle Q', R' \rangle \cap sx H_\nu x^{-1}s^{-1} = s \Big(\langle Q', R' \rangle \cap x H_\nu x^{-1}\Big) s^{-1}.
    \]
    Applying Theorem~\ref{thm:parab_join_short_conj} gives us that \(\langle Q', R' \rangle \cap x H_\nu x^{-1} = \langle Q' \cap x H_\nu x^{-1}, R' \cap x H_\nu x^{-1} \rangle\).
    Combining these two equalities gives the first case of the proposition.
    Thus we may assume \(s \notin Q' \cup R'\) for the remainder of the proof.

    If \(s \in Q\), we define \(U = s^{-1}Q'\), and otherwise set \(U = s^{-1}R'\).
    In either case let \(V = U^{-1}\).
    Throughout this proof we will assume that \(s \in Q\), with the case that \(s \in R\) being identical.
    Note that these two cases are mutually exclusive, for otherwise we would have \(s \in Q \cap R = Q' \cap R'\) by \descref{C1}, contradicting our assumption.
    
    If \(g \in Q'\) for all \(g \in \langle Q', R' \rangle \cap sx H_\nu x^{-1}s^{-1}\) with \(\abs{g}_X \geq N_1 + 2 \abs{s}_X\), then by Lemma~\ref{lem:almost_containment_inf_subgps} we have \(\langle Q', R' \rangle \cap sx H_\nu x^{-1}s^{-1} = Q' \cap sx H_\nu x^{-1}s^{-1}\) and we are done.
    Suppose to the contrary, then, that there exists some element \(g \in \langle Q', R' \rangle \cap sx H_\nu x^{-1}s^{-1}\) with \(\abs{g}_X \geq N_1 + 2\abs{s}_X\) such that \(g \notin Q'\).
    Then \(s^{-1}gs \notin s^{-1}Q's\), and so \(s^{-1}gs\) (as an element of \(U \langle Q', R' \rangle V\)) has a minimal type path representative \(p = q_1 p _1 \dots p_n q_2\) with \(n > 0\) and \(p_- = 1\). 
    Moreover, we have \(\abs{s^{-1}gs}_X \geq N_1\).
    
    Since \(x_1 H_\nu = x H_\nu\) and \(\elem{p} \in x H_\nu x^{-1}\), we have \(\elem{p} \in x_1 H_\nu x_1^{-1}\) also.
    Let \(h\) be the \(H_\nu\)-edge of \(\relcay\) with \(h_- = x_1\) and \(s^{-1}gs = \elem{p} = x_1 \elem{h} x_1^{-1}\).
    Denote the \(\Theta_1\)-shortcutting of \(p\) by \(\Sigma(p,\Theta_1) = f_0 e_1 f_1 \dots f_{m-1} e_m f_m\).
    
    By Lemma~\ref{lem:pathreps_are_tamable}, the path \(p\) is \((B_4,C'_0,\zeta,\Theta_1)\)-tamable. 
    Lemma~\ref{lem:h_conn_to_ek} tells us that \(h\) is connected to \(e_k\) for some \(k = 1, \dots, m\) and \(d_X(h_-,(e_k)_-) \leq \kappa_0\).
    Moreover, by Lemma~\ref{lem:parab_short_one_ek} \(k = m = 1\), so that \(\Sigma(p,\Theta_1) = f_0 e_1 f_1\) and 
    \begin{equation}
    \label{eq:h0_e1_close}
        d_X(h_-,(e_1)_-) \leq \kappa_0.
    \end{equation}
    Applying Lemma~\ref{lem:parab_short_no_sides}, we see that \((e_1)_-\) is a non-terminal vertex of \(q_1p_1\) and \((e_1)_+\) is a non-initial vertex of \(p_nq_2\).
    In any of the cases, \(p_1\) contains an \(H_\nu\)-component \(h'\) that is connected to \(e_1\) (and is thus, in turn, connected to \(h\)).
    
    \medskip
    \underline{\emph{Case 1:}} 
        Suppose first that \((e_1)_-\) is a vertex of \(p_1\).
        As \(h'\) is the \(H_\nu\)-component of \(p_1\) connected to \(e_1\), it must be that \((e_1)_- = h'_-\).
        By Lemma~\ref{lem:parab_short_no_sides}, we must have \(d_X(1,\elem{q_1}) = \abs{q_1}_X \leq B_4\).
        Further, \(d_X(1,h_-) = \abs{x_1}_X \leq \xi_0 \sigma^2\).
        Combining these two inequalities with (\ref{eq:h0_e1_close}), we obtain
        \begin{equation}
        \label{eq:dist_s_h'_1}
            \begin{split}
                d_X(\elem{q_1}, h'_-) &\leq d_X(\elem{q_1}, 1) + d_X(1,h_-) + d_X(h_-,h'_-)\\
                &\leq B_4 + \xi_0 \sigma^2 + \kappa_0.    
            \end{split}
        \end{equation}
    
    \medskip
    \underline{\emph{Case 2:}}
        Suppose now that \((e_1)_-\) is a non-terminal vertex of \(q_1\).
        Since \((e_1)_+\) is a vertex of either \(p_n\) or \(q_2\), \(e_1\) comes from an instance of consecutive backtracking along the segments \(q_1, p_1, \dots, p_n\) and possibly \(q_2\).
        In particular, \(q_1\) contains an \(H_\nu\)-component connected to \(h'\), the component of \(p_1\) associated with this consecutive backtracking.
        By Lemma~\ref{lem:short_cusps}, 
        \begin{equation}
        \label{eq:dist_s_h'_2}
            d_X(\elem{q_1}, h'_-) = d_X((q_1)_+, h'_-) \leq C_1.
        \end{equation}
    
    \medskip

    Since \(\elem{q_1} \in s^{-1}Q'\), there is some \(t \in Q'\) such that \(\elem{q_1} = s^{-1} t\).
    Take \(y = t^{-1} s h'_-\).
    Following (\ref{eq:dist_s_h'_1}) and (\ref{eq:dist_s_h'_2}), we have \[\abs{y}_X = d_X(s^{-1} t,h'_-) = d_X(\elem{q_1},h'_-) \leq \tau.\]
    Moreover, \(s^{-1} t y = h'_- \in x_1 H_\nu = x H_\nu\) since \(h'\) and \(h\) are connected and so \(sx \in tyH_\nu\).
    It follows that
    \begin{align*}
        t \Big(\langle Q', R' \rangle \cap y H_\nu y^{-1}\Big) t^{-1}
            &= t \langle Q', R' \rangle t^{-1} \cap t y H_\nu y^{-1} t^{-1} \\
            &= \langle Q', R' \rangle \cap s x H_\nu x^{-1} s^{-1}
    \end{align*}
    as required.
    
    Finally, note that as \(\xi_0, \sigma \geq 1\), we have \(\tau \geq \xi_0 \sigma^2 \geq \sigma\) as in the statement of the proposition.
\end{proof}


\section{Proofs of main results}
\label{sec:end}

For this section, in addition to Convention~\ref{conv:GQR}, we will suppose that \(G\) is QCERF. 
We begin with a proof of a technical intermediate to Theorem~\ref{thm:parab_join_better}.

\begin{theorem}
\label{thm:parab_join_intermediate}
    There is s finite set \(\mathcal{K}\) of maximal parabolic subgroups of \(G\) and constants \(B_5, C_5 \geq 0\) such that if \(Q' \leqslant Q\) and \(R' \leqslant R\) satisfy conditions \descref{C1}-\descref{C5} with \(B \geq B_5, C \geq C_5\), and family \(\mathcal{P} \supseteq \mathcal{P}_\tau\) (as in Notation~\ref{not:Pm}), where \(\tau\) is the constant obtained from Proposition~\ref{prop:back_to_Q'R'}, then the following is true.
    
    Suppose that \(P\) is such that \(\langle Q', R' \rangle \cap P\) infinite.
    Then there is an element \(u \in \langle Q', R' \rangle\) such that either
    \begin{enumerate}[label=(\roman*)]
        \item \(\langle Q', R' \rangle \cap P = u Q' u^{-1} \cap P\) or,
        \item \(\langle Q', R' \rangle \cap P = u R' u^{-1} \cap P\) or,
        \item \(\langle Q', R' \rangle \cap P =  u\langle Q' \cap K, R' \cap K \rangle u^{-1}\), where \(K = u^{-1} P u\) is an element of \(\mathcal{K}\).
    \end{enumerate}
    
    Moreover, if either \(Q' \cap P\) or \(R' \cap P\) is infinite, then we may take \(u = 1\) in cases (i) and (ii), and \(u \in Q' \cup R'\) in case (iii).
\end{theorem}

\begin{proof}
    We define \(B_5 = \max\{B_2(\tau),B_3,B_4\}\) and \(C_5 = \max\{C_2(1), C_3(\tau), C_4\}\), where \(B_2(\tau)\) and \(C_3(\tau)\) are the constants of Theorem~\ref{thm:parab_join_short_conj}, \(B_3\) is the constant of Lemma~\ref{lem:pass_to_q_conj}, \(C_2(1)\) is the constant from Lemma~\ref{lem:pathreps_are_tamable}, and \(B_4\) and \(C_4\) are the constants of Proposition~\ref{prop:back_to_Q'R'}.
    Take \(\mathcal{K}\) to be the set \(\{y H_\nu y^{-1} \in G \, | \, \nu \in \Nu, \abs{y}_X \leq \tau\}\).
    Let \(Q' \leqslant Q\) and \(R' \leqslant R\) be subgroups satisfying conditions \descref{C1}-\descref{C5} with constants \(B \geq B_5, C \geq C_5\), and finite family \(\mathcal{P} \supseteq \mathcal{P}_\tau\).
    Let \(P = bH_\nu b^{-1}\) be such that \(\langle Q', R' \rangle \cap P\) is infinite and \(\abs{b}_{X\cup\mathcal{H}}\) minimal among elements of \(bH_\nu\).
    
    By Lemma~\ref{lem:pass_to_q_conj}, there is \(v \in \langle Q', R' \rangle\) and \(s \in Q \cup R\) such that
    \begin{equation}
    \label{eq:int_with_parab}
        \langle Q', R' \rangle \cap P = v \Big(\langle Q', R' \rangle \cap sx H_\nu x^{-1}s^{-1}\Big) v^{-1},
    \end{equation}
    where \(\nu \in \Nu\) and \(x \in G\) with \(\abs{x}_X \leq \sigma\) and \(b = vsx\). 
    It follows that
    \begin{equation}
    \label{eq:usHsu_P}
        v s x H_\nu x^{-1} s^{-1} v^{-1} = bH_\nu b^{-1} = P.
    \end{equation}
    Note that when \(Q' \cap P\) or \(R' \cap P\) is infinite, \(v\) may be taken to be trivial.
    
    Applying Proposition~\ref{prop:back_to_Q'R'}, we have either that one of the following equations holds
    \begin{equation}
    \label{eq:int_of_conj_0}
        \langle Q', R' \rangle \cap sx H_\nu x^{-1}s^{-1} = s \langle Q' \cap x H_\nu x^{-1}, R'\cap x H_\nu x^{-1} \rangle s^{-1} \quad \textrm{with } s \in Q' \cup R',
    \end{equation}
    or that
    \begin{equation}
    \label{eq:int_of_conj_1}
        \begin{split}
            \langle Q', R' \rangle \cap s x H_\nu x^{-1} s^{-1} = Q' \cap s x H_\nu x^{-1} s^{-1} \quad \textrm{with } s \in Q, \\
            \langle Q', R' \rangle \cap s x H_\nu x^{-1} s^{-1} = R' \cap s x H_\nu x^{-1} s^{-1} \quad \textrm{with } s \in R,
        \end{split}
    \end{equation}
    or finally
    \begin{equation}
    \label{eq:int_of_conj_2}
        \langle Q', R' \rangle \cap s x H_\nu x^{-1} s^{-1} = t\Big(\langle Q' ,R' \rangle \cap y H_\nu y^{-1}\Big)t^{-1}
    \end{equation}
    where \(t \in (Q' \cup R')\), \(y \in G\) with \(\abs{y}_X \leq \tau\), and \(sx \in tyH_\nu\) so that
    \begin{equation}
    \label{eq:syHys_xHx}
        ty H_\nu y^{-1} t^{-1} = sx H_\nu x^{-1}s^{-1}.
    \end{equation}

    If (\ref{eq:int_of_conj_0}) holds, then we set \(u = vs\) and \(K = xH_\nu x^{-1}\).
    The equality 
    \[
        \langle Q', R' \rangle \cap P = u\langle Q' \cap K, R' \cap K \rangle u^{-1}
    \]
    then follows immediately from (\ref{eq:int_with_parab}).
    Observe that (\ref{eq:usHsu_P}) tells us that \(K = u^{-1}Pu\).
    Moreover, noting that \(\abs{x}_X \leq \sigma \leq \tau\), we have that \(K = xH_\nu x^{-1} \in \mathcal{K}\), as required.
    
    If instead (\ref{eq:int_of_conj_1}) holds, then from (\ref{eq:int_with_parab}) we obtain
    \[
        \langle Q', R' \rangle \cap P = v\Big(Q' \cap sx H_\nu x^{-1}s^{-1} \Big) v^{-1}
    \]
    or
    \[
        \langle Q', R' \rangle \cap P = v\Big(R' \cap sx H_\nu x^{-1}s^{-1} \Big) v^{-1},
    \]
    where in either case setting \(u = v\) gives the desired conclusion by (\ref{eq:usHsu_P}).
    
    Lastly, if (\ref{eq:int_of_conj_2}) holds, then (\ref{eq:int_with_parab}) gives that
    \begin{equation}
    \label{eq:int_parab_final}
        \langle Q', R' \rangle \cap P = v t \Big(\langle Q' ,R' \rangle \cap y H_\nu y^{-1}\Big) t^{-1} v^{-1}.
    \end{equation}
    By the choice of \(B\) and \(C\), and the fact that \(\abs{y}_X \leq \tau\), we can apply Theorem~\ref{thm:parab_join_short_conj} to obtain
    \begin{equation}
    \label{eq:appl_of_short_conj_thm}
        \langle Q' ,R' \rangle \cap y H_\nu y^{-1} = \langle Q' \cap y H_\nu y^{-1}, R' \cap y H_\nu y^{-1} \rangle.
    \end{equation}
    Combining (\ref{eq:int_parab_final}) and (\ref{eq:appl_of_short_conj_thm}) we conclude that
    \[
        \langle Q', R' \rangle \cap P = v t \langle Q' \cap K, R' \cap K \rangle t^{-1} v^{-1},
    \]
    where \(K = y H_\nu y^{-1} \in \mathcal{K}\).
    We set \(u = v t\) and note that \(u \in \langle Q', R' \rangle\), since \(t \in Q' \cup R'\).
    Since \(v = 1\) when \(Q' \cap P\) or \(R' \cap P\) is infinite, we have \(u \in Q' \cup R'\) in these cases.
    Finally, observing that (\ref{eq:usHsu_P}) and (\ref{eq:syHys_xHx}) give \(K = t^{-1} sx H_\nu x^{-1}s^{-1} t = u^{-1} P u\) completes the proof.
\end{proof}

\begin{proposition}
\label{prop:parabs_infinite_index_only}
    Suppose that either \(Q\) and \(R\) have almost compatible parabolic subgroups or that each \(H_\nu\) is double coset separable.
    Then there is a family of pairs of subgroups \(Q'\) and \(R'\) as in \descref{E} satisfying the hypotheses of Theorem~\ref{thm:parab_join_intermediate} such that the following is true.
    
    Suppose that \(P \leqslant G\) is a maximal parabolic subgroup of \(G\) with \(\langle Q', R' \rangle \cap P\) infinite and \(u \in \langle Q', R' \rangle\) is the element obtained from Theorem~\ref{thm:parab_join_intermediate}. 
    If \(S \cap u^{-1}Pu\) has finite index in \(Q' \cap u^{-1}Pu\) (respectively, in \(R' \cap u^{-1}Pu\)), then \(\langle Q', R' \rangle \cap P = uR'u^{-1} \cap P\) (respectively, \(\langle Q', R' \rangle \cap P = uQ'u^{-1} \cap P\)).
    In particular, at least one of \(Q' \cap u^{-1}Pu\) or \(R' \cap u^{-1}Pu\) is infinite.
\end{proposition}

\begin{proof}
    Let \(\mathcal{K} = \{K_1, \dots, K_n\}\) be the finite set of maximal parabolic subgroups of \(G\) provided by Theorem~\ref{thm:parab_join_intermediate}.
    If each \(H_\nu\) is double coset separable, then by Proposition~\ref{prop:sep->C1-C5}, there are subgroups \(Q' \leqslant_f Q\) and \(R' \leqslant_f R\) as in \descref{E} satisfying \descref{C1}-\descref{C5} with constants \(B_5,C_5\) (provided by Theorem~\ref{thm:parab_join_intermediate}) and finite family \(\mathcal{P}_\tau\), where \(\tau\) is the constant of Proposition~\ref{prop:back_to_Q'R'}.
    Arguing as in the proof of \cite[Theorem 14.5]{MinMin}, the same conclusion holds in the case that \(Q\) and \(R\) have almost compatible parabolics, without the double coset separability assumption.
    More precisely, there exists \(L \leqslant_f G\) with \(S \subseteq L\) such that for any \(L' \leqslant_f L\) with \(S \subseteq L'\), there is \(M \leqslant_f L'\) with \(Q \cap L' \subseteq M\) such that for any \(M' \leqslant_f M\) with \(Q \cap L' \subseteq M'\), the subgroups \(Q' = Q \cap M'\) and \(R' = R \cap M'\) satisfy these conditions.
    All such \(Q'\) and \(R'\) meet the hypotheses of Theorem~\ref{thm:parab_join_intermediate}.

    We will show that the subgroup \(L\) can be modified so that the desired conclusion holds.
    Fix some \(i = 1, \dots, n\) and note that since \(G\) is QCERF, \(Q\) and \(R\) are separable.
    Thus their intersection \(S\) is also separable.
    Whenever \(S \cap K_i \leqslant_f Q \cap K_i\), let \(U_i\) be a finite set of coset representatives of \(S \cap K_i\) in \(Q \cap K_i\), and otherwise take \(U_i\) to be the empty set.
    Similarly, whenever \(S \cap K_i \leqslant_f R \cap K_i\), let \(V_i\) be a finite set of coset representatives of \(S \cap K_i\) in \(R \cap K_i\), and otherwise take \(V_i\) to be the empty set.
    Take \(U = \bigcup_{i=1}^n (U_i \cup V_i)\), and note that \(U\) is a finite set disjoint from \(S\).
    
    Since \(S\) is separable, Lemma~\ref{lem:sep_from_finite_set} gives us \(G' \leqslant_f G\), disjoint from \(U\), with \(S \subseteq G'\).
    We take \(L_0 = L \cap G' \leqslant_f G\), noting that again \(S \subseteq L_0\) and \(L_0 \cap U = \emptyset\).
    For any \(L' \leqslant_f L_0\) with \(S \subseteq L'\), we have that \(L' \leqslant_f L\).
    Now there is \(M \leqslant_f L'\) with \(Q \cap L' \subseteq M\) as in \descref{E}. 
    Let \(M' \leqslant_f M\) be any finite index subgroup with \(Q \cap L' \subseteq M'\) and write \(Q' = Q \cap M', R' = R \cap M'\).
    By Proposition~\ref{prop:sep->C1-C5}, \(Q'\) and \(R'\) also satisfy \descref{C1}-\descref{C5}, so Theorem~\ref{thm:parab_join_intermediate} holds.

    Let \(P \leqslant G\) be a maximal parabolic subgroup of \(G\) such that \(\langle Q', R' \rangle \cap P\) is infinite, and let \(u \in \langle Q', R' \rangle\) be the element provided by Theorem~\ref{thm:parab_join_intermediate}.
    If either of the first two cases of the theorem hold, then we are done.
    Otherwise there is \(i = 1, \dots, n\) such that
    \[
        \langle Q', R' \rangle \cap P = u \langle Q' \cap K_i, R' \cap K_i \rangle u^{-1},
    \]
    with \(K_i = u^{-1} P u\).
    Suppose that \(S \cap K_i \leqslant_f Q' \cap K_i\).
    Then by the constructions of \(G'\) and \(Q'\) we have
    \[
        S \cap K_i \subseteq Q' \cap K_i = Q \cap M' \cap K_i \subseteq Q \cap G' \cap K_i = S \cap K_i,
    \]
    so that \(Q' \cap K_i = S \cap K_i\).
    It follows that \(\langle Q' \cap K_i, R' \cap K_i \rangle = R' \cap K_i\).
    Thus 
    \[
        \langle Q', R' \rangle \cap P = u\langle Q' \cap K_i, R' \cap K_i \rangle u^{-1} = u(R' \cap K_i)u^{-1} = uR'u^{-1} \cap P,
    \]
    as required.
    An identical argument (with the roles of \(Q'\) and \(R'\) swapped) gives us that \[\langle Q', R' \rangle \cap P = uQ'u^{-1} \cap P\] when \(S \cap K_i \leqslant_f R' \cap K_i\).

    To conclude, note that if \(Q' \cap u^{-1}Pu\) is finite, then \(S \cap u^{-1}Pu \leqslant_f Q' \cap u^{-1} P u\), whence \(uR'u^{-1} \cap P = \langle Q', R' \rangle \cap P\) is infinite by the hypotheses.
\end{proof}

\begin{proof}[Proof of Theorem~\ref{thm:parab_join_better}]
    Let \(\mathcal{K}\) be the finite set of maximal parabolic subgroups provided by Theorem~\ref{thm:parab_join_intermediate}.
    There is a family of pairs of finite index subgroups \(Q' \leqslant_f Q\) and \(R' \leqslant_f R\) satisfying Proposition~\ref{prop:parabs_infinite_index_only}.

    Let \(P \leqslant G\) be a maximal parabolic subgroup of \(G\) such that \(\langle Q', R' \rangle \cap P\) is infinite, and suppose that \(\langle Q', R' \rangle \cap P\) not equal to \(uQ'u^{-1} \cap P\)  or \(uR'u^{-1} \cap P\) for any \(u \in \langle Q', R' \rangle\).
    Then Theorem~\ref{thm:parab_join_intermediate} gives us \(u \in \langle Q', R' \rangle\) such that \(\langle Q', R' \rangle \cap P = u\langle Q' \cap K, R' \cap K \rangle u^{-1}\), where \(K = u^{-1}Pu\) is an element of \(\mathcal{K}\).
    Suppose that \(Q'\) and \(R'\) are almost compatible at \(K\). 
    Then either \(S \cap K \leqslant_f Q' \cap K\) or \(S \cap K \leqslant_f R' \cap K\).
    But then Proposition~\ref{prop:parabs_infinite_index_only}, gives that \(\langle Q', R' \rangle \cap P = uR'u^{-1} \cap P\) or \(\langle Q', R' \rangle \cap P = uQ'u^{-1} \cap P\) respectively.
    In either case we obtain a contradiction, completing the proof.
\end{proof}

When \(Q\) and \(R\) have almost compatible parabolics, then so do any pair of finite index subgroups \(Q' \leqslant_f Q\) and \(R' \leqslant_f R\).
It follows that the third case of Theorem~\ref{thm:parab_join_better} cannot occur for such \(Q\) and \(R\).

\begin{corollary}
\label{cor:almost_compat_parabs}
    Suppose that \(Q, R \leqslant G\) are finitely generated relatively quasiconvex subgroups with almost compatible parabolics.
    There is a family of pairs of finite index subgroups \(Q' \leqslant_f Q\) and \(R' \leqslant_f R\) as in \descref{E} with the following property.
    
    Suppose that \(P \leqslant G\) is a maximal parabolic subgroup of \(G\) with \(\langle Q', R' \rangle \cap P\) infinite.
    Then there is \(u \in \langle Q', R' \rangle\) such that \(\langle Q', R' \rangle \cap P\) is equal to either \(uQ'u^{-1} \cap P\) or \(uR'u^{-1} \cap P\).
    In particular, either \(uQ'u^{-1} \cap P\) or \(uR'u^{-1} \cap P\) is infinite.
    Moreover, if either \(Q' \cap P\) or \(R' \cap P\) is infinite, we may take \(u = 1\) in the above.
\end{corollary}

We now prove Corollaries~\ref{cor:almost_compat->virtual_compat} and \ref{thm:almost_compat_amalgamation}, with more precise existential statements than given in the introduction.
In particular, we find a finite index subgroup \(Q_1 \leqslant_f Q\) that takes over the role of \(Q\) in \descref{E} in the following. 

\begin{theorem}
\label{thm:almost_compat->virtual_compat_existence}
    Suppose that \(Q, R \leqslant G\) are finitely generated relatively quasiconvex subgroups with almost compatible parabolics.
    There is a finite index subgroup \(Q_1 \leqslant_f Q\) and a family of pairs of finite index subgroups \(Q' \leqslant_f Q_1\) and \(R' \leqslant_f R\) as in \descref{E} such that \(Q'\) and \(R'\) have compatible parabolics.    
\end{theorem}

\begin{proof}
    Suppose \(Q\) and \(R\) have almost compatible parabolics.
    By Proposition~\ref{prop:killing_finite_parabolics}, there is a finite index subgroup \(Q_1 \leqslant_f Q\) such that if \(P \leqslant G\) is a maximal parabolic subgroup of \(G\), then \(Q_1 \cap P\) is either infinite or trivial.
    Let \(Q' \leqslant_f Q_1\) and \(R' \leqslant_f R\) be finite index subgroups as in \descref{E} satisfying Corollary~\ref{cor:almost_compat_parabs}.
    Since \(Q\) and \(R\) have almost compatible parabolics, so do \(Q'\) and \(R'\).

    Let \(P \leqslant G\) be a maximal parabolic subgroup of \(G\).
    If \(Q' \cap P\) is finite, then \(Q_1 \cap P\) is finite and thus trivial by Proposition~\ref{prop:killing_finite_parabolics}.
    In this case \(Q' \cap P = \{1\} \leqslant R' \cap P\).
    On the other hand, if \(Q' \cap P\) is infinite then so is \(\langle Q', R' \rangle \cap P\).
    Now applying Corollary~\ref{cor:almost_compat_parabs}, we obtain that \(\langle Q', R' \rangle \cap P = Q' \cap P\) or \(\langle Q', R' \rangle \cap P = R' \cap P\).
    It follows that either \(R' \cap P \leqslant Q' \cap P\) or \(Q' \cap P \leqslant R' \cap P\) as required.
\end{proof}

\begin{theorem}
    Suppose that \(Q, R \leqslant G\) are finitely generated relatively quasiconvex subgroups with almost compatible parabolics.
    There is a finite index subgroup \(Q_1 \leqslant_f Q\) and a family of pairs of finite index subgroups \(Q' \leqslant_f Q_1\) and \(R' \leqslant_f R\) as in \descref{E} such that \(\langle Q',R' \rangle\) is quasiconvex and \(\langle Q',R' \rangle \cong Q' \ast_{Q' \cap R'} R'\).    
\end{theorem}

\begin{proof}
    Suppose \(Q\) and \(R\) have almost compatible parabolics and let \(Q_1 \leqslant_f Q\) be the finite index subgroup provided by Theorem~\ref{thm:almost_compat->virtual_compat_existence}.
    Note that \(S' = Q_1 \cap R\) is a fixed finite index subgroup of \(Q \cap R\) depending only on \(Q\).
    Take \(M = M(Q,R,S') \geq 0\) to be the constant of Theorem~\ref{thm:compat_parab_amalgam}.
    
    Combining Proposition~\ref{prop:sep->C1-C5} and Theorem~\ref{thm:almost_compat->virtual_compat_existence} there is a family of pairs of finite index subgroups \(Q' \leqslant_f Q_1\) and \(R' \leqslant_f R\) as in \descref{E} that have compatible parabolics and satisfy condition \descref{C2} with parameter \(B = M\).
    By Lemma~\ref{lem:C2->C2_old}, \(\minx{(Q' \cup R') \setminus S'} \geq M\).
    Note that \descref{E} ensures that \(Q' \cap R' = S'\).
    Now applying Theorem~\ref{thm:compat_parab_amalgam}, we see that \(\langle Q', R' \rangle\) is relatively quasiconvex and \(\langle Q', R' \rangle \cong Q' \ast_{Q' \cap R'} R'\) as required.
\end{proof}

\begin{proof}[Proof of Corollary~\ref{cor:full_strongly_qc_join}]
    Recall that if \(Q\) and \(R\) are strongly quasiconvex or full, they have almost compatible parabolics.
    Let \(Q\) and \(R\) be strongly relatively quasiconvex subgroups of \(G\), and let \(Q' \leqslant_f Q\) and \(R' \leqslant_f R\) be subgroups as in \descref{E} for which Theorem~\ref{thm:virtual_join_qc} and Corollary~\ref{cor:almost_compat_parabs} hold.
    Let \(P \leqslant G\) be a maximal parabolic subgroup of \(G\).
    Since \(Q\) and \(R\) have finite intersections with maximal parabolic subgroups of \(G\), so do their subgroups \(Q'\) and \(R'\).
    In particular, \(uQ'u^{-1} \cap P\) and \(uR'u^{-1} \cap P\) are finite for all \(u \in \langle Q', R' \rangle\).
    Corollary~\ref{cor:almost_compat_parabs} now directly implies that \(\langle Q', R' \rangle \cap P\) is finite.
    Therefore \(\langle Q', R' \rangle\) is strongly relatively quasiconvex.

    Now suppose that \(Q\) and \(R\) are full relatively quasiconvex subgroups, and again let \(Q' \leqslant_f Q\) and \(R' \leqslant_f R\) be subgroups as in \descref{E} for which Theorem~\ref{thm:virtual_join_qc} and Corollary~\ref{cor:almost_compat_parabs} hold.
    If \(P \leqslant G\) is a maximal parabolic subgroup of \(G\) such that \(\langle Q', R' \rangle \cap P\) is infinite, then by Corollary~\ref{cor:almost_compat_parabs}, there is \(u \in \langle Q', R' \rangle\) such that at least one of \(Q' \cap u^{-1} P u\) or \(R' \cap u^{-1} P u\) is infinite.
    Without loss of generality, say that \(R' \cap u^{-1} P u\) is infinite.
    Now \(R' \cap u^{-1} P u\) has finite index in \(R \cap u^{-1} P u\), which has finite index in \(u^{-1} P u\) since \(R\) is fully relatively quasiconvex.
    Conjugating by \(u\), we see that \(uR'u^{-1} \cap P\) has finite index in \(P\).
    Observing that \(\langle Q', R' \rangle \cap P\) contains \(uR'u^{-1} \cap P\) completes the proof.
\end{proof}

As an immediate consequence of the above, the virtual joins \(\langle Q', R' \rangle\) are hyperbolic when \(Q\) and \(R\) are strongly relatively quasiconvex.
It may be of interest that this conclusion in fact holds the under slightly weaker hypotheses.

\begin{corollary}
    Let \(G\) be a finitely generated QCERF relatively hyperbolic group.
    
    If \(Q\) is a hyperbolic relatively quasiconvex subgroup and \(R\) is a strongly relatively quasiconvex subgroup of \(G\), then there is a family of pairs of finite index subgroups \(Q' \leqslant_f Q\) and \(R' \leqslant_f R\) as in \descref{E} such that \(\langle Q', R' \rangle\) is relatively quasiconvex and hyperbolic.
\end{corollary}

\begin{proof}
    Let \(Q\) be a hyperbolic relatively quasiconvex subgroup of \(G\), and \(R\) a strongly relatively quasiconvex subgroup of \(G\).
    Since \(R\) is strongly quasiconvex, \(Q\) and \(R\) have almost compatible parabolics.
    Moreover, \(Q\) and \(R\) are hyperbolic they are certainly finitely generated, so Corollary~\ref{cor:almost_compat_parabs} applies.
    Let \(Q' \leqslant_f Q\) and \(R' \leqslant_f R\) be subgroups as in \descref{E} for which Corollary~\ref{cor:almost_compat_parabs} holds, and let \(P \leqslant G\) be a maximal parabolic subgroup of \(G\) with \(\langle Q', R' \rangle \cap P\) infinite.

    Since \(R\) is strongly relatively quasiconvex, \(uR'u^{-1} \cap P\) is finite for all \(u \in \langle Q', R' \rangle\).
    Hence Corollary~\ref{cor:almost_compat_parabs} implies that \(\langle Q', R' \rangle \cap P = uQ'u^{-1} \cap P\) for some \(u \in \langle Q', R' \rangle\).
    By Lemma~\ref{lem:props_of_qc_subgroups}, \(uQ'u^{-1}\) is relatively quasiconvex.
    Now applying Lemma~\ref{lem:hyperbolic_parabolic_quasiconvex} gives that \(uQ'u^{-1} \cap P\) is hyperbolic.
    By Hruska \cite[Theorem 9.1]{HruskaRHCG}, \(\langle Q', R' \rangle\) is hyperbolic relative to a collection of hyperbolic groups.
    Finally, \cite[Corollary 2.41]{OsinRHG} yields that \(\langle Q', R' \rangle\) is hyperbolic.
\end{proof}

\printbibliography

\end{document}